\documentclass[a4paper,11pt]{article}
\usepackage{amsmath,amsthm,amssymb,enumitem,xcolor}

\usepackage[hang]{footmisc}
\usepackage[nosort,nocompress,noadjust]{cite}

\usepackage[bookmarks=false,hyperfootnotes=false,colorlinks,
    linkcolor={red!60!black},
    citecolor={blue!50!black},
    urlcolor={blue!80!black}]{hyperref}

\renewcommand{\eqref}[1]{\hyperref[#1]{(\ref{#1})}}

\newlist{enumlist}{enumerate}{1}
\setlist[enumlist]{labelindent=0cm,label=\arabic*.,labelwidth=2.5ex,labelsep=0.5ex,leftmargin=3ex,align=left,topsep=0.5ex,itemsep=1ex,parsep=1ex}

\newlist{my-enumlist}{enumerate}{1}
\setlist[my-enumlist]{label={(\arabic*)},labelwidth=5ex,labelsep=1ex,leftmargin=6ex,labelindent=0cm,align=left,topsep=0.5ex,itemsep=1ex,parsep=1ex}

\newlist{itemlist}{itemize}{1}
\setlist[itemlist]{labelindent=0cm,label=$\bullet$,labelwidth=2.5ex,labelsep=0.5ex,leftmargin=3ex,align=left,topsep=0.5ex,itemsep=1ex,parsep=1ex}

\numberwithin{equation}{section}

{\theoremstyle{definition}\newtheorem{definition}{Definition}[section]

\newtheorem{remark}[definition]{Remark}
\newtheorem{example}[definition]{Example}}

\newtheorem{proposition}[definition]{Proposition}
\newtheorem{lemma}[definition]{Lemma}
\newtheorem{theorem}[definition]{Theorem}
\newtheorem{corollary}[definition]{Corollary}

\newtheorem{letterthm}{Theorem}

\newcommand{\bim}[3]{\mathord{\raisebox{-0.4ex}[0ex][0ex]{\scriptsize $#1$}{#2}\hspace{-0.25ex}\raisebox{-0.4ex}[0ex][0ex]{\scriptsize $#3$}}}

\newcommand{\C}{\mathbb{C}}

\newcommand{\eps}{\varepsilon}
\newcommand{\al}{\alpha}
\newcommand{\be}{\beta}

\newcommand{\Irr}{\operatorname{Irr}}

\newcommand{\ot}{\otimes}
\newcommand{\recht}{\rightarrow}

\newcommand{\Z}{\mathbb{Z}}
\newcommand{\vphi}{\varphi}
\newcommand{\op}{^\text{\rm op}}

\newcommand{\id}{\mathord{\text{\rm id}}}
\newcommand{\om}{\omega}

\newcommand{\N}{\mathbb{N}}

\newcommand{\ovt}{\mathbin{\overline{\otimes}}}

\newcommand{\Tr}{\operatorname{Tr}}

\newcommand{\Om}{\Omega}

\newcommand{\si}{\sigma}
\newcommand{\R}{\mathbb{R}}

\newcommand{\counit}{\epsilon}

\newcommand{\F}{\mathbb{F}}
\newcommand{\cH}{\mathcal{H}}

\newcommand{\cZ}{\mathcal{Z}}
\newcommand{\Ad}{\operatorname{Ad}}

\newcommand{\cK}{\mathcal{K}}
\newcommand{\cJ}{\mathcal{J}}

\newcommand{\cF}{\mathcal{F}}
\newcommand{\T}{\mathbb{T}}
\newcommand{\actson}{\curvearrowright}

\newcommand{\cW}{\mathcal{W}}
\newcommand{\cU}{\mathcal{U}}

\newcommand{\cM}{\mathcal{M}}
\newcommand{\lspan}{\operatorname{span}}
\newcommand{\cN}{\mathcal{N}}
\newcommand{\cR}{\mathcal{R}}

\newcommand{\dpr}{^{\prime\prime}}

\newcommand{\cV}{\mathcal{V}}

\newcommand{\Aut}{\operatorname{Aut}}

\newcommand{\rP}{\operatorname{P}}
\newcommand{\SO}{\operatorname{SO}}
\newcommand{\PGL}{\operatorname{PGL}}
\newcommand{\PSL}{\operatorname{PSL}}
\newcommand{\even}{_{\text{\rm even}}}
\newcommand{\odd}{_{\text{\rm odd}}}
\newcommand{\cspan}{\overline{\operatorname{span}}}
\newcommand{\zdim}{\operatorname{zdim}}

\begin{document}

\begin{center}
{\boldmath\LARGE\bf Thin II$_1$ factors with no Cartan subalgebras}

\bigskip

{\sc by Anna Sofie Krogager and Stefaan Vaes{\renewcommand\thefootnote{}\footnote{\noindent KU~Leuven, Department of Mathematics, Leuven (Belgium).\\ E-mails: annasofie.krogager@kuleuven.be and stefaan.vaes@kuleuven.be.\\ Supported by European Research Council Consolidator Grant 614195, and by long term structural funding~-- Methusalem grant of the Flemish Government.}}}
\end{center}

\begin{abstract}\noindent
It is a wide open problem to give an intrinsic criterion for a II$_1$ factor $M$ to admit a Cartan subalgebra $A$. When $A \subset M$ is a Cartan subalgebra, the $A$-bimodule $L^2(M)$ is ``simple'' in the sense that the left and right action of $A$ generate a maximal abelian subalgebra of $B(L^2(M))$. A II$_1$ factor $M$ that admits such a subalgebra $A$ is said to be $s$-thin. Very recently, Popa discovered an intrinsic local criterion for a II$_1$ factor $M$ to be $s$-thin and left open the question whether all $s$-thin II$_1$ factors admit a Cartan subalgebra. We answer this question negatively by constructing $s$-thin II$_1$ factors without Cartan subalgebras.
\end{abstract}

\section{Introduction}

One of the main decomposability properties of a II$_1$ factor $M$ is the existence of a \emph{Cartan subalgebra} $A \subset M$, i.e.\ a maximal abelian subalgebra (MASA) whose normalizer $\cN_M(A) = \{u \in \cU(M) \mid uAu^* = A\}$ generates $M$ as a von Neumann algebra. Indeed by \cite{FM75}, when $M$ admits a Cartan subalgebra, then $M$ can be realized as the von Neumann algebra $L_\Omega(\cR)$ associated with a countable equivalence relation $\cR$, possibly twisted by a scalar $2$-cocycle $\Omega$. If moreover this Cartan subalgebra is unique in the appropriate sense, this decomposition $M = L_\Omega(\cR)$ is canonical.

Although a lot of progress on the existence and uniqueness of Cartan subalgebras has been made (see e.g.\ \cite{OP07,PV11}), there is so far no intrinsic local criterion to check whether a given II$_1$ factor admits a Cartan subalgebra. When $A \subset M$ is a Cartan subalgebra, then $A \subset M$ is in particular an \emph{$s$-MASA}, meaning that the $A$-bimodule $\bim{A}{L^2(M)}{A}$ is \emph{cyclic}, i.e.\ there exists a vector $\xi \in L^2(M)$ such that $A \xi A$ spans a dense subspace of $L^2(M)$. Although it was already shown in \cite{Pu59} that the hyperfinite II$_1$ factor $R$ admits an $s$-MASA $A \subset R$ that is \emph{singular} (i.e.\ that satisfies $\cN_R(A)\dpr = A$), all examples of $s$-MASAs so far were inside II$_1$ factors that also admit a Cartan subalgebra.

Very recently in \cite{Po16}, Popa discovered that the existence of an $s$-MASA in a II$_1$ factor $M$ is an intrinsic local property. He proved that a II$_1$ factor $M$ admits an $s$-MASA if and only if $M$ satisfies the $s$-thin approximation property: for every finite partition of the identity $p_1,\ldots,p_n$ in $M$, every finite subset $\cF \subset M$ and every $\eps > 0$, there exists a finer partition of the identity $q_1,\ldots,q_m$ and a single vector $\xi \in L^2(M)$ such that every element in $\cF$ can be approximated up to $\epsilon$ in $\| \cdot \|_2$ by linear combinations of the $q_i \xi q_j$.

Although an $s$-MASA can be singular and although it is even proved in \cite[Corollary 4.2]{Po16} that every $s$-thin II$_1$ factor admits uncountably many non conjugate singular $s$-MASAs, as said above, all known $s$-thin factors so far also admit a Cartan subalgebra and Popa poses as \cite[Problem 5.1.2]{Po16} to give examples of $s$-thin factors without Cartan subalgebras. We solve this problem here by constructing $s$-thin II$_1$ factors $M$ that are even \emph{strongly solid:} whenever $B \subset M$ is a diffuse amenable von Neumann subalgebra, the normalizer $\cN_M(B)\dpr$ stays amenable. Clearly, nonamenable strongly solid II$_1$ factors have no Cartan subalgebras.

We obtain this new class of strongly solid II$_1$ factors by applying Popa's deformation/rigidity theory to Shlyakhtenko's $A$-valued semicircular systems (see \cite{Sh97} and Section \ref{Shlyakhtenko} below). When $A$ is abelian, this provides a rich source of examples of MASAs with special properties, like MASAs satisfying the $s$-thin approximation property of \cite{Po16}.

Generalizing Voiculescu's free Gaussian functor \cite{Vo83}, the data of Shlyakhtenko's construction consists of a tracial von Neumann algebra $(A,\tau)$ and a symmetric $A$-bimodule $\bim{A}{H}{A}$, where the symmetry is given by an anti-unitary operator $J : H \recht H$ satisfying $J^2=1$ and $J(a\cdot \xi \cdot b) = b^* \cdot J\xi \cdot a^*$. The construction produces a tracial von Neumann algebra $M$ containing $A$ such that $\bim{A}{L^2(M)}{A}$ can be identified with the full Fock space
$$L^2(A) \oplus \bigoplus_{n \geq 1} \bigl( \underbrace{H \ot_A \cdots \ot_A H}_{\text{$n$ times}} \bigr) \; .$$

In the same way as the free Gaussian functor transforms direct sums of real Hilbert spaces into free products of von Neumann algebras, the construction of \cite{Sh97} transforms direct sums of $A$-bimodules into free products that are amalgamated over $A$. Therefore, the deformation/rigidity results and methods for amalgamated free products introduced in \cite{IPP05,Io12}, and in particular Popa's $s$-malleable deformation obtained by ``doubling and rotating'' the $A$-bimodule,  can be applied and yield the following result, proved in Corollaries \ref{cor.rel-strong-solid} and \ref{cor.no-Cartan-weak-mixing} below (see Theorem \ref{thm.absence-Cartan} for the most general statement).

\begin{letterthm}\label{thm.main-A}
Let $(A,\tau)$ be a tracial von Neumann algebra and let $M$ be the von Neumann algebra associated with a symmetric $A$-bimodule $\bim{A}{H}{A}$. Assume that $\bim{A}{H}{A}$ is weakly mixing (Definition \ref{def2}) and that the left action of $A$ on $H$ is faithful. Then, $M$ has no Cartan subalgebra. If moreover $\bim{A}{H}{A}$ is mixing and $A$ is amenable, then $M$ is strongly solid.
\end{letterthm}


In the particular case where $A$ is diffuse abelian and the bimodule $\bim{A}{H}{A}$ is weakly mixing, we get that $A \subset M$ is a singular MASA. Very interesting examples arise as follows by taking $A = L^\infty(K,\mu)$ where $K$ is a second countable compact group with Haar probability measure $\mu$. Whenever $\nu$ is a probability measure on $K$, we consider the $A$-bimodule $H_\nu$ given by
\begin{equation}\label{eq.my-bim-compact-group}
H_\nu = L^2(K \times K,\mu \times \nu) \quad\text{with}\;\; (F \cdot \xi \cdot G)(x,y) = F(xy) \, \xi(x,y) \, G(x)  \; ,
\end{equation}
for all $F,G \in A$ and $\xi \in H_\nu$. We assume that $\nu$ is symmetric and use the symmetry
\begin{equation}\label{eq.my-symm}
J_\nu : H_\nu \recht H_\nu : (J \xi)(x,y) = \overline{\xi(xy,y^{-1})} \quad\text{for all}\;\; x,y \in K \; .
\end{equation}
We denote by $M$ the tracial von Neumann algebra associated with the $A$-bimodule $(H_\nu,J_\nu)$.

The $A$-bimodule $H_\nu$ is weakly mixing if and only if the measure $\nu$ has no atoms, while $H_\nu$ is mixing when the probability measure $\nu$ is $c_0$, meaning that the convolution operator $\lambda(\nu)$ on $L^2(K)$ is compact (see Definition \ref{def.c0} and Proposition \ref{prop.char}). So for all $c_0$ probability measures $\nu$ on $K$, we get that $M$ is strongly solid.

On the other hand, when the measure $\nu$ is concentrated on a subset of the form $F \cup F^{-1}$, where $F \subset K$ is \emph{free} in the sense that every reduced word with letters from $F \cup F^{-1}$ defines a nontrivial element of $K$, then $A \subset M$ is an $s$-MASA.

In Theorem \ref{thm.free-c0}, we construct a compact group $K$, a free subset $F \subset K$ generating $K$ and a symmetric $c_0$ probability measure $\nu$ with support $F \cup F^{-1}$. For this, we use results of \cite{AR92,GHSSV07} on the spectral gap and girth of a random Cayley graph of the finite group $\PGL(2,\Z/p\Z)$. As a consequence, we obtain the first examples of $s$-thin II$_1$ factors that have no Cartan subalgebra, solving \cite[Problem 5.1.2]{Po16}, which was the motivation for our work.

\begin{letterthm}\label{thm.main-B}
Taking a compact group $K$ and a symmetric probability measure $\nu$ on $K$ as above, the associated II$_1$ factor $M$ is nonamenable, strongly solid and the canonical subalgebra $A \subset M$ is an $s$-MASA.
\end{letterthm}

As we explain in Remark \ref{rem.bogol-crossed}, the so-called free Bogoljubov crossed products $L(\F_\infty) \rtimes G$ associated with an (infinite dimensional) orthogonal representation of a countable group $G$ can be written as the von Neumann algebra associated with a symmetric $A$-bimodule where $A = L(G)$. Therefore, our Theorem \ref{thm.main-A} is a generalization of similar results proved in \cite{Ho12b} for free Bogoljubov crossed products. Although free Bogoljubov crossed products $M = L(\F_\infty) \rtimes G$ with $G$ abelian provide examples of MASAs $L(G) \subset M$ with interesting properties (see \cite{HS09,Ho12a}), $L(G) \subset M$ can never be an $s$-MASA (see Remark \ref{rem.free-bogol-abelian}).

The point of view of $A$-valued semicircular systems is more flexible and even offers advantages in the study of free Bogoljubov crossed products $M = L(\F_\infty) \rtimes G$. Indeed, in Corollary \ref{cor.free-Bogol}, we prove that these II$_1$ factors $M$ never have a Cartan subalgebra, while in \cite{Ho12b}, this could only be proved for special classes of orthogonal representations.

In Theorem \ref{thm.max-amen}, we prove several maximal amenability results for the inclusion $A \subset M$ associated with a symmetric $A$-bimodule $(H,J)$, by combining the methods of \cite{Po83,BH16}. Again, these results generalize \cite{Ho12a,Ho12b} where the same was proved for free Bogoljubov crossed products.

We finally make some concluding remarks on the existence of $c_0$ probability measures supported on free subsets of a compact group. On an \emph{abelian} compact group $K$, a probability measure $\nu$ is $c_0$ if and only if its Fourier transform $\widehat{\nu}$ tends to zero at infinity as a function from $\widehat{K}$ to $\C$. Of course, no two elements of an abelian group are free, but the abelian variant of being free is the so-called independence property: a subset $F$ of an abelian compact group $K$ is called independent if any linear combination of distinct elements in $F$ with coefficients in $\Z \setminus \{0\}$ defines a non zero element in $K$. It was proved in \cite{Ru60} that there exist closed independent subsets of the circle group $\T$ that carry a $c_0$ probability measure. It would be very interesting to get a better understanding of which, necessarily non abelian, compact groups admit $c_0$ probability measures supported on a free subset and we conjecture that these exist on the groups $\SO(n)$, $n \geq 3$.

\section{Preliminaries}

Let $(A,\tau)$ be a tracial von Neumann algebra.

\begin{definition}\label{def1}
A \emph{symmetric} $A$-bimodule $(H,J)$ is an $A$-bimodule $\bim{A}{H}{A}$ equipped with an anti-unitary operator $J\colon H\to H$ such that $J^2=1$ and
$$J(a\cdot\xi\cdot b)=b^\ast\cdot J\xi\cdot a^\ast,\qquad\forall a,b\in A \; .$$
\end{definition}

A vector $\xi$ in a right (resp.\ left) $A$-module $H$ is said to be right (resp.\ left) bounded if there exists a $\kappa>0$ such that $\|\xi a\|\leq\kappa\|a\|_2$ (resp.\ $\|a\xi\|\leq\kappa\|a\|_2$) for all $a\in A$. Whenever $\xi$ is right bounded, we denote by $\ell(\xi)$ the map $L^2(A)\to H : a\mapsto \xi a$. Similarly, when $\xi$ is left bounded, we denote by $r(\xi)$ the map $L^2(A)\to H : a\mapsto a\xi$.

Given right bounded vectors $\xi,\eta$, the operator $\ell(\xi)^\ast\ell(\eta)$ belongs to $A$ and is denoted $\langle\xi,\eta\rangle_A$. This defines an $A$-valued scalar product associated with the right $A$-module $H$. Similarly, if $\xi,\eta\in H$ are left bounded vectors, we define an $A$-valued scalar product associated with the left $A$-module $H$ by ${}_A\langle \xi,\eta\rangle=Jr(\xi)^\ast r(\eta)J\in A$. Here, $J$ denotes the canonical involution on $L^2(A)$.

Popa's non intertwinability condition (see \cite[Section 2]{Po03}) saying that $B \not\prec_M A$ is equivalent with the existence of a sequence of unitaries $b_n \in \cU(B)$ such that $\lim_n \|E_A(x b_n y)\|_2 = 0$ for all $x,y \in M$ can be viewed as a weak mixing condition for the $B$-$A$-bimodule $\bim{B}{L^2(M)}{A}$ (cf.\ the notions of relative (weak) mixing in \cite[Definition 2.9]{Po05}). This then naturally lead to the notion of a mixing, resp.\ weakly mixing bimodule in \cite{PS12}.

\begin{definition}[{\cite{PS12}}] \label{def2} Let $(A,\tau)$ and $(B,\tau)$ be tracial von Neumann algebras and $\bim{B}{H}{A}$ a $B$-$A$-bimodule.

\begin{enumlist}
\item $\bim{B}{H}{A}$ is called \textit{left weakly mixing} if there exists a net of unitaries $b_n\in\cU(B)$ such that for all right bounded vectors $\xi,\eta\in H$, we have
$$\lim_n\|\langle b_n\xi,\eta\rangle_A\|_2=0 \; .$$

\item $\bim{B}{H}{A}$ is called \textit{left mixing} if every net $b_n\in\cU(B)$ tending to $0$ weakly satisfies $$\lim_n\|\langle b_n\xi,\eta\rangle_A\|_2=0$$ for all right bounded vectors $\xi,\eta\in H$.
\end{enumlist}
We similarly define the notions of \emph{right (weak) mixing}. When $\bim{A}{H}{A}$ is a symmetric $A$-bimodule, left (weak) mixing is equivalent with right (weak) mixing and we simply refer to these properties as (weak) mixing.
\end{definition}


In \cite[Section 2]{Po03}, Popa proved that the intertwining relation $B \prec_M A$ is equivalent with the existence of a nonzero $B$-$A$-subbimodule of $L^2(M)$ having finite right $A$-dimension. In the same way, one gets the following characterization of weakly mixing bimodules. For details, see \cite{PS12} and \cite[Theorem A.2.2]{Bo14}.

\begin{proposition}[{\cite{Po03,PS12,Bo14}}]\label{prop5}
Let $(A,\tau)$ and $(B,\tau)$ be tracial von Neumann algebras and $\bim{B}{H}{A}$ a $B$-$A$-bimodule. The following are equivalent:
\begin{enumlist}
\item $\bim{B}{H}{A}$ is left weakly mixing;
\item $\{0\}$ is the only $B$-$A$-subbimodule of $\bim{B}{H}{A}$ of finite $A$-dimension;
\item $\bim{B}{(H\otimes_A \overline{H})}{B}$ has no nonzero $B$-central vectors.
\end{enumlist}
\end{proposition}


\section{\boldmath Shlyakhtenko's $A$-valued semicircular systems}\label{Shlyakhtenko}

We first recall Voiculescu's free Gaussian functor from the category of real Hilbert spaces to the category of tracial von Neumann algebras. Let $H_\R$ be a real Hilbert space and let $H$ be its complexification. The \textit{full Fock space} of $H$ is defined as
$$ \cF(H) = \C\Omega\oplus\bigoplus_{n=1}^\infty H^{\otimes n} \; .$$
The unit vector $\Omega$ is called the \textit{vacuum vector}. Given a vector $\xi\in H$, we define the \textit{left creation operator} $\ell(\xi)\in B(\cF(H))$ by
$$
\ell(\xi)(\Omega) = \xi \quad\text{and}\quad
\ell(\xi)(\xi_1\otimes\cdots\otimes\xi_n) = \xi\otimes\xi_1\otimes\cdots\otimes\xi_n \; .
$$
Put
$$\Gamma(H_\R)'' := \{\ell(\xi)+\ell(\xi)^\ast\mid\xi\in H_\R\}'' \; .$$
This von Neumann algebra is equipped with the faithful trace given by $\tau(\cdot)=\langle \cdot\,\Omega,\Omega\rangle$.
In \cite{Vo83}, it is proved that the operator $\ell(\xi)+\ell(\xi)^\ast$ has a semicircular distribution with respect to the trace $\tau$ and that $\Gamma(H_\R)''\cong L(\F_{\dim H_\R})$. By the functoriality of the construction, any orthogonal transformation $u$ of $H_\R$ gives rise to an automorphism $\al_u$ of $\Gamma(H_\R)\dpr$ satisfying $\al_u(\ell(\xi)+\ell(\xi)^*) = \ell(u\xi) + \ell(u\xi)^*$ for all $\xi \in H_\R$. So, every orthogonal representation $\pi : G \recht O(H_\R)$ of a countable group $G$ gives rise to the \emph{free Bogoljubov action} $\si_\pi : G \actson \Gamma(H_\R)\dpr$ given by $\si_\pi(g) = \al_{\pi(g)}$ for all $g \in G$.

In \cite{Sh97}, Shlyakhtenko introduced a generalization of Voiculescu's free Gaussian functor, this time being a functor from the category of symmetric $A$-bimodules (where $A$ is any von Neumann algebra) to the category of von Neumann algebras containing $A$. We will here repeat this construction in the case where $A$ is a tracial von Neumann algebra.

Let $(A,\tau)$ be a tracial von Neumann algebra and let $(H,J)$ be a symmetric $A$-bimodule. We denote by $H^{\otimes_A^n}$ the $n$-fold Connes tensor product $H\otimes_A H\otimes_A\cdots\otimes_A H$. The full Fock space of the $A$-bimodule $\bim{A}{H}{A}$ is defined by
\begin{equation}\label{eq.full-fock}
\cF_A(H) = L^2(A)\oplus\bigoplus_{n=1}^\infty H^{\otimes_A^n} \; .
\end{equation}

We denote by $\cH$ the set of left and right $A$-bounded vectors in $H$. Since $A$ is a tracial von Neumann algebra, $\cH$ is dense in $H$.
Given a right bounded vector $\xi\in H$, we define the left creation operator $\ell(\xi)$ analogous to the case where $A=\C$ by
\begin{align*}
& \ell(\xi)(a) = \xi a,\quad a\in A \; ,\\
& \ell(\xi)(\xi_1\otimes_A\ldots\otimes_A\xi_n)=\xi\otimes_A\xi_1\otimes_A\ldots\otimes_A\xi_n,\quad \xi_i\in \cH \; .
\end{align*}
Note that $a\ell(\xi)=\ell(a\xi)$ and $\ell(\xi)a=\ell(\xi a)$ for $a\in A$ and that the adjoint map $\ell(\xi)^\ast $ satisfies
\begin{align*}
&\ell(\xi)^\ast(a)=0\quad\text{for all }a\in L^2(A) \; ,\\
&\ell(\xi)^\ast(\xi_1\otimes_A\ldots\otimes_A\xi_n)=\langle\xi,\xi_1\rangle_A\xi_2\otimes_A\ldots\otimes_A\xi_n\quad\text{for }\xi_i\in \cH \; .
\end{align*}

\begin{definition}\label{def.A-semi}
Given a tracial von Neumann algebra $(A,\tau)$ and a symmetric $A$-bimodule $(H,J)$, we consider the full Fock space $\cF_A(H)$ given by \eqref{eq.full-fock} and define
$$\Gamma(H,J,A,\tau)'':= A\vee\{\ell(\xi)+\ell(\xi)^\ast\mid\xi\in \cH,\,J\xi=\xi\}''\subset B(\cF_A(H)) \; ,$$
where $A\subset B(\cF_A(H))$ is given by the left action on $\cF_A(H)$. We also have
$$\Gamma(H,J,A,\tau)'' = A\vee\{\ell(\xi)+\ell(J\xi)^\ast\mid\xi\in\cH\}'' \; .$$
\end{definition}

We denote by $\Omega$ the vacuum vector in $\cF_A(H)$ given by $\Omega=1_A\in L^2(A)$. We define $\tau$ as the vector state on $M=\Gamma(H,J,A,\tau)''$ given by the vacuum vector $\Omega$.
Whenever $n\geq1$ and $\xi_1,\ldots,\xi_n\in\cH$, we define the Wick product as in \cite[Lemma 3.2]{HR10} by
\begin{equation}\label{eq.wick}
W(\xi_1,\ldots,\xi_n) = \sum_{i=0}^n\ell(\xi_1)\cdots\ell(\xi_i)\ell(J\xi_{i+1})^\ast\cdots\ell(J\xi_n)^\ast \; .
\end{equation}
As in \cite[Lemma 3.2]{HR10}, we get that $W(\xi_1,\ldots,\xi_n)\in M$ and
$$W(\xi_1,\ldots,\xi_n)\Omega = \xi_1\otimes_A \cdots\otimes_A \xi_n \; .$$
These elements, with $n\geq1$, span a $\|\cdot\|_2$-dense subspace of $M\ominus A$. Together with $A$, they span a $\|\cdot\|_2$-dense $*$-subalgebra of $M$.

\begin{proposition}[{\cite{Sh97}}]\label{prop1}
The state $\tau(\cdot)=\langle\cdot\,\Omega,\Omega\rangle$ defined above is a faithful trace on $M$.
\end{proposition}

\begin{proof}
Define $\cJ\colon \cF_A(H)\to \cF_A(H)$ by $\cJ(a)=a^\ast$ for $a\in A$ and
$$\cJ(\xi_1\otimes_A \cdots \otimes_A \xi_n)=J\xi_n\otimes_A \cdots \otimes_A J\xi_1$$
for $\xi_1,\ldots,\xi_n\in \cH$. Then $\cJ$ is an anti-unitary map satisfying $\cJ^2=1$. One easily checks that $\cJ\ell(\xi)\cJ = r(J\xi)$ for all $\xi\in \cH$ and that $\cJ a\cJ$ is just right multiplication by $a^\ast$ on $\cF_A(H)$. This implies that $\cJ M\cJ$ commutes with $M$. Indeed, for $\xi,\eta\in\cH$ with $J\xi=\xi$ and $J\eta=\eta$, we have $\langle\xi,a\eta\rangle_A = {}_A\langle\xi a,\eta\rangle$ since
\begin{align*}
\langle Jr(\xi a)^\ast r(\eta) J x,y\rangle &= \langle r(\xi a) y^\ast , r(\eta) x^\ast\rangle = \langle y^\ast\xi a, x^\ast\eta \rangle = \langle J(x^\ast\eta), J(y^\ast\xi a)\rangle \\
&= \langle \eta x, a^\ast\xi y\rangle = \langle \ell(\xi)^\ast \ell(a\eta) x,y\rangle \; ,
\end{align*}
for all $x,y\in A$. It follows that

\begin{align*}
(\ell(\xi)^\ast r(\eta) + \ell(\xi)r(\eta)^\ast)(a) = \langle \xi,a\eta\rangle_A = {}_A\langle\xi a,\eta\rangle = (r(\eta)^\ast\ell(\xi) + r(\eta)\ell(\xi)^\ast)(a),\quad \forall a\in A \; .
\end{align*}
Since $\ell(\xi)$ and $r(\eta)^\ast$ clearly commute when restricted to $\cF_A(H)\ominus L^2(A)$, it follows that $\ell(\xi)+\ell(\xi)^\ast$ commutes with $r(\eta)+r(\eta)^\ast$. We conclude that $M$ commutes with $\cJ M\cJ$.

Next, we show that $\cJ(x\Omega) = x^\ast\Omega$ for all $x\in M$. This clearly holds for $x\in A$ so it suffices to prove it for $x$ of the form $x=W(\xi_1,\ldots,\xi_n)$ with $\xi_i\in\cH$. We have
\begin{align*}
\cJ(W(\xi_1,\ldots,\xi_n)\Omega) &= \cJ(\xi_1\otimes_A \cdots\otimes_A \xi_n) = J\xi_n\otimes_A \cdots\otimes_A J\xi_1 \\ &= W(J\xi_n,\ldots,J\xi_1)\Omega = W(\xi_1,\ldots,\xi_n)^\ast\Omega \; .
\end{align*}
We now get that
\begin{align*}
\tau(xy) &= \langle xy\Omega,\Omega\rangle = \langle x\cJ(y^\ast\Omega),\Omega\rangle = \langle x\cJ y^\ast\cJ\Omega,\Omega\rangle = \langle \cJ y^\ast\cJ x\Omega,\Omega\rangle\\
&= \langle x\Omega,\cJ y\cJ\Omega\rangle = \langle x\Omega,y^\ast\Omega\rangle = \langle yx\Omega,\Omega\rangle = \tau(yx) \; ,
\end{align*}
for all $x,y\in M$ and hence $\tau$ is a trace.

It is easy to check that $\Omega\in\cF_A(H)$ is a cyclic vector for both $M$ and $\cJ M\cJ$. Hence $\Omega$ is also separating for $M$ and it follows that $\tau$ is faithful.
\end{proof}

By construction, we have that $L^2(M)\cong\cF_A(H)$ as $A$-bimodules.

In \cite{Sh97}, Shlyakhtenko used the terminology \textit{$A$-valued semicircular system} for the family $\{\ell(\xi)+\ell(\xi)^\ast\mid\xi\in \cH, J\xi=\xi\}$, as an analogue to the free Gaussian functor case, where the operator $\ell(\xi)+\ell(\xi)^\ast$ has a semicircular distribution with respect to $\tau$.

\begin{example}\phantomsection\label{ex2}
\begin{enumlist}
\item When $H=L^2(A)$ is the trivial $A$-bimodule with $J(a) = a^*$, we simply get
$$\Gamma(H,J,A,\tau)''=A \ovt L^\infty[0,1] \; .$$
Indeed, $A$ commutes with $\ell(\hat{1})+\ell(\hat{1})^\ast$ and they together generate $\Gamma(H,J,A,\tau)''$. In particular, we see that $\Gamma(H,J,A,\tau)''$ is not always a factor.

\item When $H=L^2(A)\otimes L^2(A)$ is the coarse $A$-bimodule with $J(a \ot b) = b^* \ot a^*$, we get
$$\Gamma(H,J,A,\tau)''=(A,\tau)\ast L^\infty[0,1] \; .$$
This example shows that the construction of $\Gamma(H,J,A,\tau)''$ may depend on the trace on $A$. Indeed, if $A=\C^2$ we can consider the trace $\tau_\delta$ for any $\delta\in(0,1)$ given by $\tau_\delta(a,b) = \delta a + (1-\delta)b$, $a,b\in\C$. By \cite[Lemma 1.6]{Dy92} we have that $L(\Z)\ast (A,\tau_\delta)=L(\F_{1+2\delta(1-\delta)})$, the interpolated free group factor. It is wide open whether the interpolated free group factors are all isomorphic. So at least, there is no obvious isomorphism between $\Gamma(H,J,A,\tau_{\delta_1})''$ and $\Gamma(H,J,A,\tau_{\delta_2})''$ for $\delta_1\neq\delta_2$. In Example \ref{exam.factor}, we shall actually see that even the factoriality of $\Gamma(H,J,A,\tau)''$ may depend on the choice of the trace $\tau$. For a general factoriality criterion for $\Gamma(H,J,A,\tau)\dpr$, see Theorem \ref{thm.absence-Cartan}.
\end{enumlist}
\end{example}

Note that the construction of $\Gamma(H,J,A,\tau)''$ is functorial in the following sense. If $U\in\cU(H)$ is a unitary operator that is $A$-bimodular and commutes with $J$, then $U$ defines a trace-preserving automorphism of $M=\Gamma(H,J,A,\tau)''$ in the following way. Since $U$ is $A$-bimodular, we can define a unitary $U^n$ on $H^{\otimes_A^n}$ by $U^n(\xi_1\otimes_A\cdots\otimes_A\xi_n) = (U\xi_1\otimes_A\cdots\otimes_A U\xi_n)$. The direct sum of these unitaries (and the identity on $L^2(A)$) then gives an $A$-bimodular unitary operator on $\cF_A(H)$, which we will still denote by $U$. Note that $U\ell(\xi)U^\ast=\ell(U\xi)$ for all $\xi\in \cH$. Since $U$ commutes with $J$, it follows that $UMU^\ast=M$ so that $\Ad U$ defines an automorphism of $M$.

Recall that for Voiculescu's free Gaussian functor, we have that the direct sum of Hilbert spaces translates into the free product of von Neumann algebras, in the sense that $\Gamma(H_1\oplus H_2) = \Gamma(H_1)\ast\Gamma(H_2)$. In the setting of $A$-bimodules in general, we instead get the amalgamated free product over $A$ as stated in the following proposition.

\begin{proposition}[{\cite[Proposition 2.17]{Sh97}}] \label{prop2}
Let $(H_1,J_1)$ and $(H_2,J_2)$ be symmetric $A$-bi\-mod\-ules. Put $H=H_1\oplus H_2$ and $J=J_1\oplus J_2$. Then
$$\Gamma(H,J,A,\tau)''\cong\Gamma(H_1,J_1,A,\tau)''\ast_A\Gamma(H_2,J_2,A,\tau)'' \; ,$$
with respect to the unique trace-preserving conditional expectation onto $A$.
\end{proposition}

\begin{remark}\label{rem.bogol-crossed}
As we recalled in the beginning of this section, to every orthogonal representation $\pi : G \recht O(K_\R)$ of a countable group $G$ on a real Hilbert space $K_\R$ is associated the free Bogoljubov action $\si_\pi : G \actson \Gamma(K_\R)\dpr$. Write $A = L(G)$ and equip $A$ with its canonical tracial state $\tau$. Denote by $K$ the complexification of $K_\R$ and define the symmetric $A$-bimodule $\bim{A}{H}{A}$ given by
\begin{equation}\label{eq.bimodule-bogol}
\begin{split}
H = \ell^2(G) \ot K \quad & \text{with}\quad u_g \cdot (\delta_h \ot \xi) \cdot u_k = \delta_{ghk} \ot \pi(g) \xi \\ & \text{and}\quad J(\delta_h \ot \xi) = \delta_{h^{-1}} \ot \pi(h^{-1}) \overline{\xi}
\end{split}
\end{equation}
where $(\delta_g)_{g \in G}$ denotes the canonical orthonormal basis of $\ell^2(G)$. It is now straightforward to check that there is a canonical trace preserving isomorphism
$$\Gamma(H,J,A,\tau)\dpr \cong \Gamma(K_\R)\dpr \rtimes^{\si_\pi} G$$
that maps $A$ onto $L(G)$ identically.
\end{remark}

\begin{example}\label{exam.factor}
This final example illustrates that even the factoriality of $\Gamma(H,J,A,\tau)\dpr$ may depend on the choice of $\tau$. 
Take $A = \C^2$, $\al \in \Aut(A)$ the flip automorphism and $H = \C^2$ with $A$-bimodule structure given by $a \cdot \xi \cdot b = \al(a) \xi b$. Define $J : H \recht H : J(a) = \al(a)^*$. The $n$-fold tensor power $H^{\ot_A^n}$ can be identified with $\C^2$ with the bimodule structure given by
$$a \cdot \xi \cdot b = \begin{cases} a \xi b &\;\;\text{if $n$ is even,} \\ \al(a) \xi b &\;\;\text{if $n$ is odd.}\end{cases}$$
We denote by $\{e_n,f_n\}$ the canonical orthonormal basis of $H^{\ot_A^n}$ under this identification. For every $0 < \delta < 1$, denote by $\tau_\delta$ the trace on $A$ given by $\tau_\delta(a,b) = \delta a + (1-\delta)b$. With respect to the canonical trace $\tau = \tau_{1/2}$, the left and right creation operators associated with the identity $1 \in A = H$ then become
$$\ell(e_n) = e_{n+1} \;\; , \;\; \ell(f_n) = f_{n+1} \;\; , \;\; r(e_n) = f_{n+1} \;\; , \;\; r(f_n) = e_{n+1} \;\;,$$
for all $n \geq 0$.

By symmetry, it suffices to consider the case $0 < \delta \leq 1/2$. With respect to the trace $\tau_\delta$, the left and right creation operators $\ell_\delta$ and $r_\delta$ can be realized on the same Hilbert space and are given by $$\ell_\delta = \ell \, \lambda(D^{-1/2}) \quad\text{and}\quad r_\delta = r \, \rho(D^{-1/2}) \;\;,$$
where $D = (2\delta,2(1-\delta))$ is the Radon-Nikodym derivative between $\tau_\delta$ and $\tau_{1/2}$ and where we denote by $\lambda(\,\cdot\,)$ and $\rho(\,\cdot\,)$ the left, resp.\ right, action of $A$. Then,
$$M_\delta := \Gamma(H,J,A,\tau_\delta)\dpr = \lambda(A) \vee \{\ell_\delta + \ell_\delta^*\}\dpr = \lambda(A) \vee \{S_\delta\}\dpr \; \; ,$$
where $S_\delta = \ell \lambda(\Delta^{-1/4}) + \ell^* \lambda(\Delta^{1/4})$ and $\Delta = (\delta/(1-\delta),(1-\delta)/\delta)$. We still denote by $\tau_\delta$ the canonical trace on $M_\delta$.

Note that $S_\delta = S_\delta^*$. Denoting by $e=(1,0)$ and $f = (0,1)$ the minimal projections in $A$, we have that $S_\delta e = f S_\delta$. When $\delta = 1/2$, the operator $S_\delta$ is nonsingular and diffuse. When $0 < \delta < 1/2$, the kernel of $S_\delta$ has dimension $1$ and $S_\delta$ is diffuse on its orthogonal complement. We denote by $z_\delta$ the projection onto the kernel of $S_\delta$. Then $z_\delta$ is a minimal and central projection in $M_\delta$ with $\tau_\delta(z_\delta) = 1-2\delta$. We conclude that there is a trace preserving $*$-isomorphism
\begin{equation}\label{eq.my-iso}
(M_\delta,\tau_\delta) \cong \; \underbrace{M_2(\C) \ot B}_{\delta (\Tr \ot \tau_0)} \; \oplus \;  \underbrace{\C}_{1-2\delta} \;\;
\end{equation}
where $(B_0,\tau_0)$ is a diffuse abelian von Neumann algebra with normal faithful tracial state $\tau_0$ and where we emphasized the choice of trace at the right hand side. Under the isomorphism \eqref{eq.my-iso}, we have that
$$e \mapsto (e_{11} \ot 1) \oplus 0 \;\;,\;\; f \mapsto (e_{22} \ot 1) \oplus 1 \;\;,\;\; S_\delta \mapsto ((e_{12} + e_{21}) \ot b) \oplus 0 \;\;,\;\; z_\delta \mapsto 0 \oplus 1 \;\;$$
where $b \in B$ is a positive nonsingular element generating $B$.

Next, taking $H \oplus H$ and $J \oplus J$, it follows from Proposition \ref{prop2} that
$$\cM_\delta := \Gamma(H \oplus H,J \oplus J,A,\tau_\delta)\dpr = M_\delta \ast_A M_\delta \;\;,$$
where we used at the right hand side the amalgamated free product w.r.t.\ the unique $\tau_\delta$-preserving conditional expectations. We denote with superscripts $^{(1)}$ and $^{(2)}$ the elements of $M_\delta$ viewed in the first, resp.\ second copy of $M_\delta$ in the amalgamated free product. Note that $f^{(1)} = f^{(2)}$ and that, denoting this projection as $f$, we get that $f M_\delta^{(1)} f$ and $f M_\delta^{(2)} f$ are free inside $f \cM_\delta f$. It then follows from \cite{Vo86} that the projection $z := z_\delta^{(1)} \wedge z_\delta^{(2)}$ is nonzero if and only if $\delta < 1/3$. Using the diffuse subalgebras $B^{(1)}$ and $B^{(2)}$, we get that $\cZ(\cM_\delta) = \C z + \C (1-z)$. We conclude that $\Gamma(H \oplus H,J \oplus J,A,\tau_\delta)\dpr$ is a factor if and only if $1/3 \leq \delta \leq 2/3$.
\end{example}

\section{\boldmath Normalizers and (relative) strong solidity}

The main result of this section is the following dichotomy theorem for $A$-valued semicircular systems. In the special case of free Bogoljubov crossed products (see Remark \ref{rem.bogol-crossed}), this result was proven in \cite[Theorem B]{Ho12b}. As explained in the introduction, the $A$-valued semicircular systems fit perfectly into Popa's deformation/rigidity theory. The proof of Theorem \ref{thm1} therefore follows closely \cite{IPP05,HS09,HR10,Io12,Ho12b}, using in the same way Popa's $s$-malleable deformation given by ``doubling and rotating'' the initial $A$-bimodule $\bim{A}{H}{A}$ (see below).

We freely use Popa's intertwining-by-bimodules (see \cite[Section 2]{Po03}) and the notion of relative amenability (see \cite[Section 2.2]{OP07}).

\begin{theorem}\label{thm1}
Let $(A,\tau)$ be a tracial von Neumann algebra and $(H,J)$ a symmetric $A$-bimodule. Put $M = \Gamma(H,J,A,\tau)\dpr$. Let $q\in M$ be a projection and $B\subset qMq$ a von Neumann subalgebra. If $B$ is amenable relative to $A$, then at least one of the following statements holds: $B\prec_M A$ or $\cN_M(B)''$ stays amenable relative to $A$.
\end{theorem}

As a consequence of Theorem \ref{thm1}, we get the following strong solidity theorem.

\begin{corollary}\label{cor.rel-strong-solid}
Let $(A,\tau)$ be a tracial von Neumann algebra and $(H,J)$ a symmetric $A$-bimodule. Denote $M = \Gamma(H,J,A,\tau)\dpr$. Assume that $\bim{A}{H}{A}$ is mixing.

If $B \subset M$ is a diffuse von Neumann subalgebra that is amenable relative to $A$, then $\cN_M(B)\dpr$ stays amenable relative to $A$.

So if $A$ is amenable and $\bim{A}{H}{A}$ is mixing, we get that $M$ is strongly solid.
\end{corollary}
\begin{proof}
Denote $P:=\cN_M(B)''$. Since $B \vee (B' \cap M) \subset P$, we have $P' \cap M = \cZ(P)$. Denote by $z \in \cZ(P)$ the smallest projection such that $P z \not\prec_M A$. Then, $P(1-z)$ fully embeds into $A$ inside $M$ and, in particular, $P(1-z)$ is amenable relative to $A$. It remains to prove that also $Pz$ is amenable relative to $A$.

Since the bimodule $\bim{A}{H}{A}$ is mixing, the inclusion $A \subset M$ is mixing in the sense of \cite[Proof of Theorem 3.1]{Po03} and \cite[Definition 9.2]{Io12}. Since $\cN_{zMz}(Bz)\dpr = Pz$, since $Bz$ is diffuse and since $Pz \not\prec_M A$, it follows from \cite[Lemma 9.4]{Io12} that $Bz \not\prec_M A$. It then follows from Theorem \ref{thm1} that $Pz$ is amenable relative to $A$.
\end{proof}

To prove Theorem \ref{thm1}, we fix a tracial von Neumann algebra $(A,\tau)$ and a symmetric $A$-bimodule $(H,J)$. Put $M=\Gamma(H,J,A,\tau)''$ as in Definition \ref{def.A-semi}. Recall that $L^2(M)=\cF_A(H)=L^2(A)\oplus\bigoplus_{n=1}^\infty H^{\otimes_A^n}$.

We construct as follows an $s$-malleable deformation of $M$ in the sense of \cite{Po03}. Put
$$\cM = \Gamma(H\oplus H, A, J\oplus J)'' \; .$$
By Proposition \ref{prop2}, we have $\cM = M\ast_A M$. We denote by $\pi_1$ and $\pi_2$ the two canonical embeddings of $M$ into $\cM$. When no embedding is explicitly mentioned, we will always consider $M\subset\cM$ via the embedding $\pi_1$.

Let $U_t\in\cU(H\oplus H)$, $t\in\R$, be the rotation with angle $t$, i.e.,
$$U_t(\xi,\eta)=(\cos(t)\xi-\sin(t)\eta, \sin(t)\xi+\cos(t)\eta)\quad\text{for }\xi,\eta\in H \; .$$
Since the construction of $\Gamma(H,J,A,\tau)''$ is functorial, this gives rise to an automorphism $\theta_t:=\Ad U_t\in\Aut(\cM)$. Note that $\theta_{\pi/2}\circ\pi_1=\pi_2$.

Define $\beta\in\cU(H)$ by $\beta(\xi,\eta)=(\xi,-\eta)$ for $\xi,\eta\in H$. Again by functoriality, we have that $\beta$ defines an automorphism of $\cM$. Now, $\beta$ satisfies $\beta(x)=x$ for all $x\in\pi_1(M)$, $\beta^2=\id$ and $\beta\circ\theta_t=\theta_{-t}\circ\beta$ for all $t$. Hence $(\cM,(\theta_t)_{t\in\R})$ is an $s$-malleable deformation of $M$.

The following two lemmas are the key ingredients in the proof of Theorem \ref{thm1}.

\begin{lemma}\label{lem1}
Let $q\in M$ be a projection and $P\subset qMq$ a von Neumann subalgebra. If $\theta_t(P)\prec_\cM\pi_i(M)$ for some $i\in\{1,2\}$ and some $t\in(0,\frac{\pi}{2})$, then $P\prec_M A$.
\end{lemma}

\begin{lemma}\label{lem2}
Let $q\in M$ be a projection and $P\subset qMq$ a von Neumann subalgebra. If $\theta_t(P)$ is amenable relative to $A$ inside $\cM$ for all $t\in(0,\frac{\pi}{2})$, then $P$ is amenable relative to $A$ inside $M$.
\end{lemma}

Before proving Lemma \ref{lem1} and Lemma \ref{lem2}, we first show how Theorem \ref{thm1} follows from these two lemmas and we deduce a relative strong solidity theorem for $A$-valued semicircular systems.

\begin{proof}[Proof of Theorem \ref{thm1}]
Put $P=\cN_{qMq}(B)''$. We apply \cite[Theorem A]{Va13} to the subalgebra $\theta_t(B)\subset M\ast_A M$ for a fixed $t\in(0,\frac{\pi}{2})$. Note that $\theta_t(B)$ is normalized by $\theta_t(P)$. So, we get that one of the following holds:
\begin{enumlist}
\item $\theta_t(B)\prec_{\cM} A$.
\item $\theta_t(P)\prec_{\cM} \pi_i(M)$ for some $i\in\{1,2\}$.
\item $\theta_t(P)$ is amenable relative to $A$ inside $\cM$.
\end{enumlist}
If 1 or 2 holds, it follows by Lemma \ref{lem1} that $B\prec_M A$. So, if we assume that $B\nprec_M A$, we get that $\theta_t(P)$ is amenable relative to $A$ inside $\cM$ for all $t\in(0,\frac{\pi}{2})$. It then follows from Lemma \ref{lem2} that $P=\cN_{qMq}(B)''$ is amenable relative to $A$ inside $M$.
\end{proof}

\subsection*{Proof of Lemma \ref{lem1}}

We now turn to the proof of Lemma \ref{lem1}. We first give a sketch of the proof. For each $k\in\N$, we let $p_k\in B(L^2M)$ denote the projection onto $H^{\otimes_A^k}$. Given a von Neumann subalgebra $P\subset qMq$, we first show that if $\theta_t(P)\prec_\cM \pi_i(M)$ for some $i\in\{1,2\}$ and some $t\in(0,\frac{\pi}{2})$, then $P$ has ``bounded tensor length", in the sense that there exists $k\in\N$ and $\delta>0$ such that $\|\sum_{i=0}^k p_i(a)\|_2\geq\delta$ for all $a\in\cU(P)$ (see Lemma \ref{lem3}). Next, we reason exactly as in the proof of \cite[Theorem 4.1]{Po03}. Since $\theta_t$ converges uniformly to $\id$ on the unit ball of $p_i(M)$ for any fixed $i\in\N$, we get a $t\in(0,\frac{\pi}{2})$ and a nonzero partial isometry $v\in\cM$ such that $\theta_t(a)v=va$ for all $a\in\cU(P)$. Using the automorphism $\beta$, we can even obtain $t=\pi/2$, i.e., $\pi_2(a)v=v\pi_1(a)$ for all $a\in\cU(P)$. Using results of \cite{IPP05} on amalgamated free product von Neumann algebras, this implies that $P\prec_M A$.

For simplicity, we put $M_i=\pi_i(M)\subset\cM$ for $i\in\{1,2\}$. Note that
$$ L^2(M_1) = L^2(A)\oplus\bigoplus_{k=1}^\infty (H\oplus 0)^{\otimes_A^k},\quad L^2(M_2) = L^2(A)\oplus\bigoplus_{k=1}^\infty (0\oplus H)^{\otimes_A^k} \; ,$$
as subspaces of $L^2(\cM)=\cF_A(H\oplus H)$. Denote by $e_{M_i}\in B(L^2(\cM))$ the projection onto $L^2(M_i)$.

\begin{lemma}\label{lem4}
If $\mu_n\in L^2(M_1)$ is a bounded sequence such that $\lim_{n\to\infty}\|p_k(\mu_n)\|=0$ for all $k\geq 0$, then for all $i=1,2$, $0<t<\frac{\pi}{2}$, integers $a,b,c,d\geq0$ and vectors $\xi_i,\eta_i,\gamma_i,\rho_i\in \cH\oplus\cH$, we have
$$ \lim_{n\to\infty} \|e_{M_i}(\ell(\xi_1)\cdots\ell(\xi_a)\ell(\eta_b)^\ast\cdots\ell(\eta_1)^\ast r(\gamma_c)\cdots r(\gamma_1)r(\rho_1)^\ast\cdots r(\rho_d)^\ast U_t\mu_n)\| = 0 \; .$$
\end{lemma}

\begin{proof}
Fix $t\in(0,\frac{\pi}{2})$ and define $\delta_1=\cos t$ and $\delta_2=\sin t$. Define the operator $Z_i\in B(L^2\cM)$ for $i=1,2$ by
$$ Z_i = \bigoplus_{e\geq b+d}\delta_i^{e-b-d}(U_t^{\otimes_A^b} \; \otimes_A \; 1^{\otimes_A^{(e-b-d)}} \; \otimes_A \; U_t^{\otimes^d_A}) \; .$$
Denote $p_{\geq\kappa} = \sum_{i=\kappa}^\infty p_i$ and $p_{<\kappa}=\sum_{i=0}^{k-1}p_i$. When $\kappa \geq b+d$, we have $\|Z_i p_{\geq \kappa}\| = \delta_i^{\kappa - b - d}$. Since $\lim_n \|p_{< \kappa}(\mu_n)\|=0$ for every $\kappa$, we get that $\lim_n \|Z_i (\mu_n)\| = 0$. So, it suffices to prove that
\begin{align*}
&e_{M_i}(\ell(\xi_1)\cdots\ell(\xi_a)\ell(\eta_b)^\ast\cdots\ell(\eta_1)^\ast r(\gamma_c)\cdots r(\gamma_1)r(\rho_1)^\ast\cdots r(\rho_d)^\ast U_tp_{\geq b+d}(\mu)) \\
&\qquad= \ell(q_i \xi_1)\cdots\ell(q_i \xi_a)\ell(\eta_b)^\ast\cdots\ell(\eta_1)^\ast r(q_i \gamma_c)\cdots r(q_i \gamma_1)r(\rho_1)^\ast\cdots r(\rho_d)^\ast Z_i(\mu)
\end{align*}
for all $\mu\in L^2(M_1)$, where $q_1$, resp. $q_2$, denotes the orthogonal projection of $H \oplus H$ onto $H \oplus 0$, resp.\ $0 \oplus H$. It is sufficient to check this formula for $\mu=\mu_1\otimes_A \cdots \otimes_A \mu_e$ with $\mu_i\in \cH\oplus 0$ and $e\geq b+d$, where it follows by a direct computation.
\end{proof}

\begin{lemma}\label{lem3}
If $a_n\in M$ is a bounded sequence with $\lim_n\|p_k(a_n)\|_2=0$ for all $k\geq 0$, then
$$\lim_{n\to\infty}\|E_{M_i}(x\theta_t(a_n)y)\|_2=0 \; ,$$
for all $i\in\{1,2\}$, $0<t<\frac{\pi}{2}$ and $x,y\in\cM$.
\end{lemma}

\begin{proof}
It suffices to take $x=W(\xi_1,\ldots,\xi_k)$ and $y=W(\eta_1,\ldots,\eta_m)$ with $\xi_i,\eta_i\in\cH\oplus\cH$ (as defined in Section \ref{Shlyakhtenko}), since these elements span a $\|\cdot\|_2$-dense subspace of $\cM\ominus A$. Then,
\begin{align*}
& E_{M_i}(x\theta_t(a_n)y) = e_{M_i}(xJy^\ast JU_t(a_n\Omega)) \\
& = \sum_{s=0}^k\sum_{r=0}^m e_{M_i}(\ell(\xi_1)\cdots\ell(\xi_s)\ell(J\xi_{s+1})^\ast\cdots\ell(J\xi_k)^\ast r(\eta_m)\cdots r(\eta_{r+1})r(J\eta_{r})^\ast\cdots r(J\eta_1)^\ast U_t(a_n\Omega)) \; ,
\end{align*}
and the result now follows from Lemma \ref{lem4}
\end{proof}

We are now ready to finish the proof of Lemma \ref{lem1}.

\begin{proof}[Proof of Lemma \ref{lem1}]
Assume that $\theta_t(P)\prec M_i$ for some $i\in\{1,2\}$ and $t\in(0,\frac{\pi}{2})$. By Lemma \ref{lem3}, we get a $\delta>0$ and $\kappa>0$ such that $\|\sum_{i=0 }^\kappa p_i(a)\|^2_2\geq2\delta$ for all $a\in\cU(P)$. Note that $\langle U_t(p_i(a)),p_j(a)\rangle = 0$ if $i\neq j$ and that $\langle U_t(p_i(a)), p_i(a)\rangle = \cos(t)^i\|p_i(a)\|_2^2$. Choose $t_0\in(0,\frac{\pi}{2})$ such that $\cos(t_0)^i\geq 1/2$ for all $i=0,\ldots,\kappa$. Note that we may choose $t_0$ of the form $t_0=\pi/2^n$. For all $a\in\cU(P)$, we then have
\begin{align*}
\tau(\theta_{t_0}(a)a^\ast) &= \langle U_{t_0}(a),a\rangle = \sum_{i,j=0}^\infty\langle U_{t_0}(p_i(a)),p_j(a)\rangle= \sum_{i=0}^\infty\cos(t_0)^i\|p_i(a)\|_2^2\\
&\geq \sum_{i=0}^\kappa\cos(t_0)^i\|p_i(a)\|_2^2\geq\frac{1}{2} 2\delta= \delta\; .
\end{align*}

Let $v$ be the unique element of minimal $\|\cdot\|_2$-norm in the $\|\cdot\|_2$-closed convex hull of $\{\theta_{t_0}(a)a^\ast\mid a\in\cU(P)\}$. Then $v\in\cM$ and $\theta_{t_0}(a)v=va$ for all $a\in\cU(P)$. Moreover, $v\neq0$ since $\tau(v)\geq\delta$.

Put $w_1=\theta_{t_0}(v\beta(v^\ast))$. Then $w_1$ satisfies $w_1a = \theta_{2t_0}(a)w_1$ for all $a\in\cU(P)$. However, we do not know yet that $w_1$ is nonzero. Assuming that $P\nprec_M A$, we have from Proposition \ref{prop2} and \cite[Theorem 1.2.1]{IPP05} that $P'\cap q\cM q\subset qMq$, hence $v^\ast v\in qMq$. Thus
$$ w_1^\ast w_1 = \theta_{t_0}(\beta(v)v^\ast v\beta(v^\ast)) = \theta_{t_0}(\beta(vv^\ast))\neq 0 \; . $$
By iterating this process, we obtain $w=w_{n-1}\neq0$ such that $wa=\theta_{\pi/2}(a)w$, i.e., $w\pi_1(a)=\pi_2(a)w$ for all $a\in P$. This means that $P \prec_\cM M_2$. As in \cite[Claim 5.3]{Ho07}), this is incompatible with our assumption $P \not\prec_M A$. So it follows that $P \prec_M A$ and the lemma is proved.
\end{proof}

\subsection*{Proof of Lemma \ref{lem2}}

\begin{proof}
Let $P\subset qMq$ and assume that $\theta_t(P)$ is amenable relative to $A$ in $\cM$ for all $t\in(0,\frac{\pi}{2})$. As in the proof of \cite[Theorem 5.1]{Io12} (and \cite[Theorem 3.4]{Va13}), we let $I$ be the set of all quadruples $i=(X,Y,\delta,t)$ where $X\subset\cM$ and $Y\subset\cU(P)$ are finite subsets, $\delta\in(0,1)$ and $t\in(0,\frac{\pi}{2})$. Then $I$ is a directed set when equipped with the ordering $(X,Y,\delta,t)\leq(X',Y',\delta',t')$ if and only if $X\subset X'$, $Y\subset Y'$, $\delta'\leq\delta$ and $t'\leq t$.

By \cite[Theorem 2.1]{OP07}, we can for each $i=(X,Y,\delta,t)\in I$ choose a vector $\xi_i\in \theta_t(q)L^2(\cM)\otimes_A L^2(\cM)\theta_t(q)$ such that $\|\xi_i\|_2\leq1$ and
\begin{align*}
&|\langle x\xi_i,\xi_i\rangle-\tau(x\theta_t(q))|\leq\delta \quad \text{for every }x\in X\text{ or }x=(\theta_t(y)-y)^\ast(\theta_t(y)-y)\text{ with }y\in Y \; ,\\
&\|\theta_t(y)\xi_{i}-\xi_{i}\theta_t(y)\|_2\leq\delta \quad \text{for every }y\in Y \; .
\end{align*}


We now prove that $\bim{qMq}{L^2(qMq)}{P}$ is weakly contained in $\bim{qMq}{(qL^2(\cM)\otimes_A L^2(\cM)q)}{P}$. For this, it suffices to show that
\begin{equation}\label{eq.suff-weak-contain}
\begin{split}
&\lim_i\langle x\xi_i,\xi_i\rangle = \tau(x) \quad \text{for every }x\in q\cM q \; ,\\
&\lim_i\|y\xi_i-\xi_iy\|_2=0 \quad \text{for every }y\in P \; .
\end{split}
\end{equation}
Let $y\in\cU(P)$ and $\eps>0$ be given. Choose $t>0$ small enough so that $\|\theta_t(y)-y\|_2^2\leq\eps/6$. We have
\begin{align*}
\|y\xi_i-\xi_i y\|_2 &\leq \|(y-\theta_t(y))\xi_i\|_2 + \|\theta_t(y)\xi_i-\xi_i\theta_t(y)\|_2 + \|\xi_i(\theta_t(y)-y)\|_2
\end{align*}
for any $i\in I$. Moreover,
$$ \|(y-\theta_t(y))\xi_i\|_2^2=\langle(\theta_t(y)-y)^\ast(\theta_t(y)-y)\xi_i,\xi_i\rangle \leq \|(\theta_t(y)-y)\theta_t(q)\|_2^2+\eps/6\leq \eps/3 \; , $$
for $i\geq(\{0\},\{y\},\eps/6,t)$ in $I$. Similarly, we get that $\|\xi_i(\theta_t(y)-y)\|_2\leq\eps/3$. Thus, we conclude that $\|y\xi_i-\xi_iy\|_2\leq\eps$ for $i\geq(\{0\},\{y\},\eps/6,t)$ and so the second assertion of \eqref{eq.suff-weak-contain} holds true. The first assertion is proved similarly, using that $\|\theta_t(q)-q\|_2\to0$ as $t\to0$.

By Proposition \ref{prop2}, we have $\cM=M_1\ast_A M_2$. Under our identification $M = M_1$, we then get that $\bim{M}{L^2(\cM)}{A} \cong \bim{M}{(L^2(M) \ot_A \cK)}{A}$, where $\bim{A}{\cK}{A}$ is the $A$-bimodule defined as the direct sum of $L^2(A)$ and all alternating tensor products $L^2(M_2 \ominus A) \ot_A L^2(M_1 \ominus A) \ot_A \cdots$ starting with $L^2(M_2 \ominus A)$. We conclude that $\bim{qMq}{L^2(qMq)}{P}$ is weakly contained in $\bim{qMq}{(q L^2(M) \ot_A (\cK \ot_A L^2(\cM)q))}{P}$. It then follows from \cite[Proposition 2.4]{PV11} that $P$ is amenable relative to $A$ inside $M$. This finishes the proof of Lemma \ref{lem2}.
\end{proof}

\section{\boldmath Maximal amenability}

%

Fix a tracial von Neumann algebra $(A,\tau)$ and a symmetric Hilbert $A$-bimodule $\bim{A}{H}{A}$ with symmetry $J : H \recht H$. Denote by $M = \Gamma(H,J,A,\tau)\dpr$ the associated von Neumann algebra with faithful normal tracial state $\tau$. We prove the following maximal amenability property by combining Popa's asymptotic orthogonality \cite{Po83} with the method of \cite{BH16}. In the special case of free Bogoljubov crossed products (see Remark \ref{rem.bogol-crossed}), part 3 of Theorem \ref{thm.max-amen} was proved in \cite[Theorem D]{Ho12b}.

\begin{theorem}\label{thm.max-amen}
Assume that $\bim{A}{H}{A}$ is weakly mixing. Then the following properties hold.
\begin{enumlist}
\item $\cZ(M) = \{a \in \cZ(A) \mid a \xi = \xi a \;\;\text{for all}\;\; \xi \in H \}$.
\item If $B \subset M$ is a von Neumann subalgebra that is amenable relative to $A$ inside $M$ and if the bimodule $\bim{B \cap A}{H}{A}$ is left weakly mixing, then $B \subset A$.
\item A von Neumann subalgebra of $M$ that properly contains $A$ is not amenable relative to $A$ inside $M$. If the $A$-bimodule $\bim{A}{H}{A}$ is faithful\footnote{A $P$-$Q$-bimodule $\bim{P}{H}{Q}$ is called faithful if the $*$-homomorphisms $P \recht B(H)$ and $Q\op \recht B(H)$ are faithful.}, then $M$ has no amenable direct summand. If $A$ is amenable, then $A \subset M$ is a maximal amenable subalgebra.
\end{enumlist}
\end{theorem}

\begin{proof}
As above, identify
$$L^2(M) = L^2(A) \oplus \bigoplus_{n \geq 1} \bigl(\underbrace{H \ot_A \cdots \ot_A H}_{\text{$n$-fold}}\bigr)$$
and denote by $\cH \subset H$ the subspace of vectors that are both left and right bounded.

1.\ Since $\bim{A}{H}{A}$ is weakly mixing, it follows from Proposition \ref{prop5} that the $n$-fold tensor products $H \ot_A \cdots \ot_A H$ (with $n \geq 1$) have no $A$-central vectors. Therefore, $A' \cap M = \cZ(A)$. Looking at the commutator of $a \in \cZ(A)$ and $\ell(\xi) + \ell(J\xi)^*$, the conclusion follows.

2.\ Since $B$ is amenable relative to $A$ inside $M$, we can fix a $B$-central state $\om \in \langle M,e_A \rangle^*$ such that $\om|_M = \tau$.

{\bf Claim I.} For every $\xi \in \cH$ and every $\eps > 0$, there exists a projection $p \in A$ such that $\tau(1-p)<\eps$ and such that
$$\om(\ell(\xi p) \ell(\xi p)^*) < \eps \; .$$
To prove this claim, fix $\xi \in \cH$ and $\eps > 0$. Define $a = \sqrt{\langle \xi,\xi\rangle_A}$ and denote by $q \in A$ the support projection of $a$. Take a projection $q_1 \in qAq$ that commutes with $a$, such that $\tau(q-q_1) < \eps/2$ and such that $a q_1$ is invertible in $q_1 A q_1$. Denote by $b \in q_1 A q_1$ this inverse and define $\eta = \xi b$. By construction, $\ell(\eta)^* \ell(\eta) = q_1$ and $\xi q_1 = \eta a$.

Pick a positive integer $N$ such that $2^{-N} < \eps / (2 \|a\|^2)$. Put $\kappa = 2^N$. Then pick $\delta > 0$ such that $\delta < \eps / (\kappa 2 \|a\|^2)$. We start by constructing unitary operators $v_1,\ldots,v_\kappa \in \cU(A \cap B)$ and a projection $q_2 \in q_1 A q_1$ such that $\tau(q_1 - q_2) < \eps/2$ and such that the vectors $\eta_i = v_i \eta$ satisfy
\begin{equation}\label{eq.goal1}
\| q_2 \langle \eta_i , \eta_j \rangle_A \, q_2 \| < \delta \quad\text{whenever}\;\; i \neq j
\end{equation}
(and where we indeed use the operator norm at the left hand side of \eqref{eq.goal1}).

We put $e_0 = q_1$ and $v_1 = 1$. Denoting by $(a_i)$ the net of unitaries in $B \cap A$ witnessing the left weak mixing of $\bim{B \cap A}{H}{A}$, we get that $\lim_i \| \langle \eta , a_i \eta \rangle_A \|_2 = 0$. So we find a net of projections $r_i \in q_1 A q_1$ such that $\tau(q_1 - r_i) \recht 0$ and
$$\| r_i \, \langle \eta , a_i \eta \rangle_A \, r_i \| < \delta \quad\text{for every $i$.}$$
Take $i$ large enough such that $\tau(q_1 - r_i) < \eps/4$ and define $e_1 := r_i$ and $v_2 := a_i$. We have now constructed $v_1,v_2$. Inductively, we double the length of the sequence, until we arrive at $v_1,\ldots,v_\kappa$. After $k$ steps, we have constructed the projections $e_1 \geq \cdots \geq e_k$ and unitaries $v_1,\ldots,v_{2^k}$ in $\cU(B \cap A)$ such that $\tau(e_{j-1} - e_j) < 2^{-j-1} \eps$ and such that the vectors $\eta_i = v_i \eta$ satisfy
$$\|e_k \, \langle \eta_i , \eta_j \rangle_A \, e_k \| < \delta \quad\text{whenever}\;\; i \neq j \; .$$
As in the first step, we can pick a unitary $a \in \cU(B \cap A)$ and a projection $e_{k+1} \in e_k A e_k$ such that $\tau(e_k-e_{k+1}) < 2^{-k-2} \eps$ and such that
$$\|e_{k+1} \, \langle \eta_i , a \eta_j \rangle_A \, e_{k+1} \| < \delta$$
for all $i,j \in \{1,\ldots,2^k\}$. It now suffices to put $v_{2^k+i} = a v_i$ for all $i = 1,\ldots,2^k$. We have doubled our sequence. We continue for $N$ steps and put $q_2 = e_N$. So, \eqref{eq.goal1} is proved.

Put $\mu_i = \eta_i q_2 = v_i \eta q_2$. Define the projections $P_i = \ell(\mu_i) \ell(\mu_i)^*$ and note that $P_i = v_i P_1 v_i^*$. By construction, $\|P_i P_j\| < \delta$ whenever $i \neq j$. Writing $P = \sum_{i=1}^\kappa P_i$ it follows that $\|P^2 - P \| < \kappa^2 \delta$. Since $P$ is a positive operator, we conclude that $\|P\| < 1 + \kappa^2 \delta$. Since $\om$ is $B$-central and $v_i \in B$ for all $i$, we get that
$$\kappa \, \om(P_1) = \sum_{i=1}^\kappa \om(P_i) = \om(P) \leq \|P\| < 1 + \kappa^2 \delta \; .$$
Therefore, $\om(P_1) < \kappa^{-1} + \kappa \delta < \|a\|^{-2} \eps$.

Since $q_1$ and $a$ commute, the right support of $(q_1-q_2)a$ is a projection of the form $q_1 - p_0$ where $p_0 \in q_1 A q_1$ is a projection with $\tau(q_1-p_0) \leq \tau(q_1-q_2) < \eps/2$. By construction, $q_1 a p_0 = q_2 a p_0$. Since  $p_0 \leq q_1$ and $\eta = \eta q_1$, it follows that
$$\xi p_0 = \xi q_1 p_0 = \eta a p_0 = \eta q_1 a p_0 = \eta q_2 a p_0 \; .$$
Define the projection $p \in A$ given by $p = (1-q) + p_0$. Since $\xi(1-q) = 0$, we still have $\xi p = \eta q_2 a p_0$. Because $1-p = (q-q_1) + (q_1-p_0)$, we get that $\tau(1-p) < \eps$. Finally,
$$\om(\ell(\xi p) \ell(\xi p)^*) = \om(\ell(\eta q_2) \, a p_0 a^* \, \ell(\eta q_2)^*) \leq \|a\|^2 \, \om(\ell(\eta q_2) \ell(\eta q_2)^*) = \|a\|^2 \, \om(P_1) < \eps \; .$$
So, we have proven Claim~I.

{\bf Claim~II.} For every $\xi \in \cH$ and every $\eps > 0$, there exists a projection $p \in A$ such that $\tau(1-p)<\eps$ and such that $\om(\ell(\xi p) \ell(\xi p)^*) = 0$.

For every integer $k \geq 1$, Claim~I gives a projection $p_k \in A$ with $\tau(1-p_k) < 2^{-k}\eps$ and $\om(\ell(\xi p_k) \ell(\xi p_k)^*) < 1/k$. Defining $p = \bigwedge_k p_k$, we get that $\tau(1-p) < \eps$ and, for every $k \geq 1$,
$$\om(\ell(\xi p) \ell(\xi p)^*) = \om(\ell(\xi) p \ell(\xi)^*) \leq \om(\ell(\xi) p_k \ell(\xi)^*) = \om(\ell(\xi p_k) \ell(\xi p_k)^*) < 1/k \; .$$
So, $\om(\ell(\xi p) \ell(\xi p)^*) = 0$ and claim~II is proved.

We can now conclude the proof of 2. Denote by $E_A : M \recht A$ and $E_B : M \recht B$ the unique trace preserving conditional expectations. It is sufficient to prove that $E_B \circ E_A = E_B$. So we have to prove that $E_B(x) = 0$ for all $x \in M \ominus A$. Using the Wick products defined in \eqref{eq.wick}, it suffices to prove that $E_B(W(\xi_1,\ldots,\xi_k))=0$ for all $k \geq 1$ and all $\xi_1,\ldots,\xi_k \in \cH$.

Since $\om$ is $B$-central and $\om|_M = \tau$, there is a unique conditional expectation $\Phi : \langle M,e_A \rangle \recht B$ such that $\Phi|_M = E_B$ and $\om = \tau \circ \Phi$.

We first consider $k \geq 2$ and $\xi_1,\ldots,\xi_k \in \cH$. By Claim~II, we can take sequences of projections $p_n,q_n \in A$ such that $p_n \recht 1$ and $q_n \recht 1$ strongly and
$$\Phi(\ell(\xi_1 p_n)\ell(\xi_1 p_n)^*) = 0 = \Phi(\ell((J\xi_k)q_n) \ell((J \xi_k)q_n)^*)$$
for all $n$. Then also $\Phi(\ell(\xi_1 p_n) T) = 0 = \Phi(T \ell((J \xi_k)q_n)^*)$ for all $n$ and all $T \in \langle M,e_A \rangle$. We conclude that
$$E_B(W(\xi_1 p_n,\xi_2,\ldots,\xi_{k-1}, q_n \xi_k)) = \Phi(W(\xi_1 p_n,\xi_2,\ldots,\xi_{k-1}, q_n \xi_k)) = 0$$
for all $n$. Since $E_B$ is normal, it follows that $E_B(W(\xi_1,\ldots,\xi_k))=0$.

We next consider the case $k=1$. So it remains to prove that $E_B(\ell(\xi) + \ell(J\xi)^*) = 0$ for all $\xi \in \cH$. For this, it is sufficient to prove that $\Phi(\ell(\xi))=0$ for all $\xi \in \cH$. By Claim~II and reasoning as above, we find a sequence of projections $p_n \in A$ such that $p_n \recht 1$ strongly and $\Phi(\ell(\xi p_n) T)=0$ for all $n$ and all $T \in \langle M,e_A \rangle$. In particular, we can take $T=1$ and get that $\Phi(\ell(\xi) p_n) = 0$ for all $n$. Write $e_n = 1-p_n$. Then,
$$\Phi(\ell(\xi))^* \Phi(\ell(\xi)) = \Phi(\ell(\xi) e_n)^* \Phi(\ell(\xi) e_n) \leq \|\ell(\xi)\|^2 \, \Phi(e_n) = \|\ell(\xi)\|^2 \, E_B(e_n) \; .$$
Since $E_B(e_n) \recht 0$ strongly, we conclude that $\Phi(\ell(\xi)) = 0$. This concludes the proof of 2.

3. It follows from 2 that a von Neumann subalgebra of $M$ properly containing $A$ is not amenable relative to $A$ and thus, not amenable itself. Whenever $H \neq \{0\}$, we have $A \neq M$ and we conclude that $M$ is not amenable. By 1, any direct summand of $M$ is given as the von Neumann algebra associated with the symmetric weakly mixing $Az$-bimodule $H z$ where $z \in \cZ(A)$ is a nonzero central projection satisfying $\xi z = z \xi$ for all $\xi \in H$. If $\bim{A}{H}{A}$ is faithful, we have $Hz \neq \{0\}$ and it follows that this direct summand is not amenable. The final statement is an immediate consequence of 2.
\end{proof}

\section{\boldmath Absence of Cartan subalgebras}

In this section, we give a complete description of the structure of the von Neumann algebra $M = \Gamma(H,J,A,\tau)\dpr$ associated with an arbitrary symmetric $A$-bimodule $(H,J)$. We describe the trivial direct summands of $M$ and then prove that the remaining direct summand never has a Cartan subalgebra and describe its center (see Theorem \ref{thm.absence-Cartan}). In all interesting cases, there are no trivial direct summands and this allows us to prove absence of Cartan subalgebras whenever $H$ is a weakly mixing $A$-bimodule (Corollary \ref{cor.no-Cartan-weak-mixing}), when $A$ is a II$_1$ factor and $H$ is not the trivial bimodule nor the bimodule given by a period $2$ automorphism of $A$ (Corollary \ref{cor.no-Cartan-A-factor}), and finally for arbitrary free Bogoljubov crossed products (Corollary \ref{cor.free-Bogol}). This last result improves \cite[Corollary C]{Ho12b}.


To prove our general structure theorem, we need the following terminology. Fix a tracial von Neumann algebra $(A,\tau)$. We say that an $A$-bimodule $H$ is given by a partial automorphism if one of the following two equivalent conditions holds.
\begin{itemlist}
\item The commutant of the right $A$ action on $H$ equals the left $A$ action, and vice versa.
\item There exists a projection $e \in B(\ell^2(\N)) \ovt A$, a central projection $z \in \cZ(A)$ and a surjective $*$-isomorphism $\al : Az \recht e (B(\ell^2(\N)) \ovt A)e$ such that $\bim{A}{H}{A} \cong e(\ell^2(\N) \ot L^2(A))$ with the bimodule structure given by $a \cdot \xi \cdot b = \al(a) \xi b$.
\end{itemlist}

Fix a symmetric $A$-bimodule $(H,J)$ and denote $M = \Gamma(H,J,A,\tau)\dpr$. Then, $M$ has two trivial direct summands. First denote by $z_0 \in \cZ(A)$ the largest projection such that $z_0 H = \{0\}$. Then, $z_0 \in \cZ(M)$ and $M z_0 = A z_0$. Next, there is a largest projection $z_1 \in \cZ(A)(1-z_0)$ such that $z_1 H = H z_1$ and such that the $A$-bimodule $H z_1$ is given by a partial automorphism of $A$ (see Lemma \ref{lem.unique-dec} for details). Again $z_1 \in \cZ(M)$ and $M z_1$ can be computed by the methods of Example \ref{exam.factor}. In a way, $Mz_1$ is not very interesting, since it is always a direct sum of a corner of $A$ and a corner of $A \ovt L^\infty([0,1])$ or of an index $2$ extension of this.

Writing $z_2 = 1-(z_0+z_1)$, we thus get that
$$M = A z_0 \oplus \Gamma(Hz_1,J,Az_1,\tau)\dpr \oplus \Gamma(Hz_2,J,Az_2,\tau)\dpr$$
and only the third direct summand is ``interesting and nontrivial''. By Lemma \ref{lem.unique-dec}, the symmetric $A z_2$-bimodule $H z_2$ is \emph{completely nontrivial} in the following sense: the left action of $Az_2$ on $H$ is faithful and there are no nonzero projections $e,f \in \cZ(A)z_2$ such that $e H = H f$ and such that $e H$ is given by a partial automorphism of $Az_2$. So it suffices to describe the structure of the von Neumann algebra associated with an arbitrary completely nontrivial symmetric $A$-bimodule.

We denote by $\dim_{-A}(K)$ the right $A$-dimension of a right Hilbert $A$-module $K$. Recall that the value of $\dim_{-A}(K)$ depends on the choice of the trace $\tau$. We similarly define $\dim_{A-}(K)$ for a left Hilbert $A$-module $K$. As in \eqref{eq.def-Delta}, for every $A$-bimodule $H$, there is a unique element $\Delta^\ell_H$ in the extended positive part of $\cZ(A)$ characterized by $\tau(\Delta^\ell_H e) = \dim_{-A}(eH)$ for every projection $e \in \cZ(A)$.


\begin{theorem}\label{thm.absence-Cartan}
Let $(A,\tau)$ be a tracial von Neumann algebra and $(H,J)$ a completely nontrivial symmetric $A$-bimodule. Write $M = \Gamma(H,J,A,\tau)\dpr$. There is a canonical central projection $q \in \cZ(M)$ (which, most of the time, is zero) such that the following holds.
\begin{my-enumlist}[label={(\alph*)}]
\item\label{stat-a} No direct summand of $M(1-q)$ is amenable relative to $A(1-q)$.
\item\label{stat-b} No direct summand of $M(1-q)$ admits a Cartan subalgebra.
\item\label{stat-c} $M q = A q$ and the support of $E_A(1-q)$ equals $1$.
\item\label{stat-d} Defining $C := \{a \in \cZ(A) \mid a \xi = \xi a \;\text{for all}\;\xi \in H\}$, we get that $\cZ(M) = \cZ(A)q + C(1-q)$.
\end{my-enumlist}
Moreover, we have that $E_A(q) = Z(\Delta^\ell_H)$, where $Z : (0,+\infty) \recht \R$ is the positive function given by $Z(t)=1-t$ when $t \in (0,1)$ and $Z(t)=0$ when $t \geq 1$.
%
%
\end{theorem}


\begin{corollary}\label{cor.no-Cartan-weak-mixing} 
Let $(A,\tau)$ be a tracial von Neumann algebra and $(H,J)$ a symmetric $A$-bimodule. Put $M=\Gamma(H,J,A,\tau)''$. If $\bim{A}{H}{A}$ is weakly mixing and faithful, then no direct summand of $M$ has a Cartan subalgebra and $\cZ(M) = \{a \in \cZ(A) \mid a \xi = \xi a \;\;\text{for all}\;\; \xi \in H\}$.
\end{corollary}


\begin{proof}
%
Let $z \in \cZ(A)$ be a nonzero central projection. Since $zH \neq \{0\}$ and $zH$ is still left weakly mixing as an $A$-bimodule, we have that $\dim_{-A}(zH) = +\infty$ and that $zH$ is not given by a partial automorphism of $A$. So the conclusions follow from Theorem \ref{thm.absence-Cartan}.
\end{proof}

When $A$ is a II$_1$ factor, the results of Theorem \ref{thm.absence-Cartan} can be formulated more easily as follows.

\begin{corollary}\label{cor.no-Cartan-A-factor}
Let $A$ be a II$_1$ factor with its unique tracial state $\tau$ and let $(H,J)$ be a symmetric $A$-bimodule. Denote $M = \Gamma(A,\tau,H,J)\dpr$. Unless $H$ is zero or $H$ is the trivial $A$-bimodule or $H$ is the symmetric $A$-bimodule associated with a period $2$ outer automorphism of $A$, the following holds: $M$ is a factor, $M$ is not amenable relative to $A$ and $M$ has no Cartan subalgebra.
\end{corollary}
\begin{proof}
Since $A$ is a II$_1$ factor, the only symmetric $A$-bimodules given by a partial automorphism of $A$ are the trivial $A$-bimodule and the $A$-bimodule given by $\al \in \Aut(A)$ with $\al \circ \al$ being inner. When a symmetric $A$-bimodule $H$ is not given by a partial automorphism of $A$, we have that $\dim_{-A}(H) > 1$. So, the conclusion follows from Theorem \ref{thm.absence-Cartan}.
\end{proof}

We finally deduce that free Bogoljubov crossed products never have a Cartan subalgebra. In \cite[Corollary C]{Ho12b}, this was proven under extra assumptions on the underlying orthogonal representation.

\begin{corollary}\label{cor.free-Bogol}
Let $G$ be an arbitrary countable group and $\pi : G \recht O(K_\R)$ an orthogonal representation of $G$ with $\dim(K_\R) \geq 2$. Denote by $\si_\pi : G \actson \Gamma(K_\R)\dpr \cong L(\F_{\dim K_\R})$ the associated free Bogoljubov action with crossed product $M := \Gamma(K_\R)\dpr \rtimes^{\si_\pi} G$ (see Remark \ref{rem.bogol-crossed}). Then no direct summand of $M$ has a Cartan subalgebra. Also, $M$ is a factor if and only if $\pi(g) \neq 1$ for every $g \in G \setminus \{e\}$ that has a finite conjugacy class.
\end{corollary}
\begin{proof}
Write $A = L(G)$ with its canonical tracial state $\tau$. By Remark \ref{rem.bogol-crossed}, we can view $M = \Gamma(H,J,A,\tau)\dpr$ where the symmetric $A$-bimodule $(H,J)$ is given by \eqref{eq.bimodule-bogol}. Denote by $K$ the complexification of $K_\R$. Observe that $H \cong \ell^2(G) \ot K$ with bimodule structure $a \cdot \xi \cdot b = \al(a) \xi b$, where $\al : L(G) \recht L(G) \ovt B(K)$ is given by $\al(u_g) = u_g \ot \pi(g)$ for all $g \in G$. Since $(\tau \ot \id)\al(a) = \tau(a)1$ for all $a \in L(G)$, it follows that $\Delta^\ell_H = \dim(K_\R) \, 1$.

The left and right actions of $A$ on $H$ are faithful. Since $H \ot_A \overline{H}$ can be identified with the bimodule associated with the representation $\pi \ot \overline{\pi}$, the center valued dimension of $H \ot_A \overline{H}$ as a left $A$-module equals $\dim(K_\R)^2 \, 1$. It follows from Lemma \ref{lem.char-bimod-partial-aut} below that $H$ is completely nontrivial. So, all conclusions follow from Theorem \ref{thm.absence-Cartan}.
%
%
%
\end{proof}

We now prove Theorem \ref{thm.absence-Cartan}, using several lemmas that we prove at the end of this section.

\begin{proof}[Proof of Theorem \ref{thm.absence-Cartan}]
Let $K \subset H$ be the maximal left weakly mixing $A$-subbimodule of $H$, i.e.\ the orthogonal complement of the span of all $A$-subbimodules of $H$ having finite right $A$-dimension. Denote by $z_0 \in \cZ(A)$ the support of the left $A$ action on $K$. In the first part of the proof, assuming $z_0 \neq 0$, we show that
\begin{my-enumlist}
\item\label{my-one} $\cZ(M)z_0 \subset \cZ(A)z_0$,
\item\label{my-two} every $M$-central state $\omega$ on $\langle M,e_A \rangle$ that is normal on $M$ satisfies $\omega(z_0)=0$.
\end{my-enumlist}

Note that $K \subset z_0H$. Denote by $\cK \subset K$ the dense subspace of vectors that are both left and right bounded. Define the von Neumann subalgebra $N \subset z_0Mz_0$ given by
\begin{equation}\label{eq.vnalg-N}
N := \bigl(A z_0 \cup \{W(\xi,J(\mu)) \mid \xi,\mu \in \cK\}\bigr)\dpr \; ,
\end{equation}
where we used the notation of \eqref{eq.wick}. Then, the linear span of $Az_0$ and elements of the form $W(\xi_1,J(\mu_1),\ldots,\xi_k,J(\mu_k))$, $k \geq 1$, $\xi_i,\mu_i \in \cK$, is a dense $*$-subalgebra of $N$.

Whenever $K_1,\ldots,K_n \subset H$ are $A$-subbimodules, we denote by concatenation $K_1\cdots K_n$ the $A$-subbimodule of $L^2(M)$ given by
$$K_1\cdots K_n := K_1 \ot_A \cdots \ot_A K_n \subset H \ot_A \cdots \ot_A H \subset L^2(M) \; .$$
In the same way, we write powers of $A$-subbimodules and when $K_i \subset H^{k_i}$ are $A$-subbimodules, then $K_1 \cdots K_n \subset H^{k_1+\cdots+k_n}$ is a well defined $A$-subbimodule.

Using this notation, note that $L^2(N)$ is the direct sum of $L^2(Az_0)$ and the spaces $L_n := (K \, J(K))^n$, $n \geq 1$. Since $K$ is a left weakly mixing $A$-bimodule, it follows that $N \cap (Az_0)' = \cZ(A)z_0$.

We claim that
\begin{my-enumlist}[resume]
\item\label{NA-weak-mix} $N \not\prec_N Az_0$, meaning that the $N$-$A$-bimodule $L^2(N)$ is left weakly mixing.
\end{my-enumlist}
Since $N \cap (Az_0)' = \cZ(A)z_0$, to prove this claim, it suffices to show that $\dim_{-A}(L^2(N)e)=+\infty$ for every nonzero projection $e \in \cZ(A)z_0$. Since the left action of $Az_0$ on $K$ is faithful and $K$ is left weakly mixing, we get that $\dim_{-A}(K \, J(K) \, e) = +\infty$. So certainly $\dim_{-A}(L^2(N)e)=+\infty$ and the claim follows.

Proof of \ref{my-one}. Define the $A$-subbimodule $R \subset L^2(M)$ given as
$$
R := \bigl( \, H \ominus \overline{(K + J(K))} \, \bigr) \oplus \bigoplus_{n=0}^\infty (H \ominus K) \, H^n \, (H \ominus J(K)) \; .
$$
Since $K$ is left weakly mixing and $J(K)$ is right weakly mixing, all $A$-central vectors in $L^2(M)$ belong to $L^2(A) + R$. Next note that left, resp.\ right multiplication by elements of $N$ induces an $N$-bimodular unitary operator
$$L^2(N) \ot_A R \ot_A L^2(N) \recht \overline{N R N} \subset L^2(z_0Mz_0) \; .$$
Since the $N$-$A$-bimodule $L^2(N)$ is left weakly mixing, it follows that $\overline{N R N}$ has no nonzero $N$-central vectors. Every element $x \in \cZ(M)z_0$ defines a vector in $L^2(z_0Mz_0)$ that is both $A$-central and $N$-central. By $A$-centrality, we conclude that $x \in Az_0 + z_0 R z_0$. In particular, $x \in L^2(N) + \overline{N R N}$. Since $x$ is $N$-central and $\overline{NRN}$ has no nonzero $N$-central vectors, we get that $x \in L^2(N)$ and thus, $x \in \cZ(A)z_0$.

Proof of \ref{my-two}.
%
%
Denote $L\even := L^2(N)$ and define $L\odd$ as the direct sum of the $A$-bimodules $(K \, J(K))^n \, K$, $n \geq 0$. Note that both $L\even$ and $L\odd$ are $N$-$A$-bimodules. The same argument as in the proof of Theorem \ref{thm.max-amen}, using the left weak mixing of $K$, shows that the von Neumann algebras $B(L\even) \cap (A\op)'$ and $B(L\odd) \cap (A\op)'$ admit no $N$-central states that are normal on $N$. Note that we have the following decomposition of $L^2(z_0M)$ as an $N$-$A$-bimodule:
$$L^2(z_0M) = \Bigl(L\even \ot_A \Bigl( L^2(A) \oplus \bigoplus_{n \geq 0} (H \ominus K) \, H^n \Bigr) \Bigr) \oplus \Bigl( L\odd \ot_A \Bigl( L^2(A) \oplus \bigoplus_{n \geq 0} (H \ominus J(K)) \, H^n \Bigr)\Bigr) \; .$$
This decomposition induces $*$-homomorphisms from $B(L\even) \cap (A\op)'$ and $B(L\odd) \cap (A\op)'$ to $B(z_0 L^2(M)) \cap (A\op)' = z_0 \langle M,e_A \rangle z_0$. So, $z_0 \langle M,e_A \rangle z_0$ admits no $N$-central state that is normal on $N$. A fortiori, \ref{my-two} holds.

Next we define the projection $z_1 \in \cZ(A)(1-z_0)$ given by
\begin{equation}\label{eq.def-z1}
z_1 = 1_{(1,+\infty]}\bigl(\Delta^\ell_{(1-z_0)H}\bigr) \; .
\end{equation}
We also write $z = z_0+z_1$ and $z_2 = 1-z$.

Denote by $e' \in \cZ(A) z_1$ the maximal projection with the following properties: the right support $f \in \cZ(A)$ of $e'H$ satisfies $e'H = zHf$ and the $A$-bimodule $e'H$ is given by a partial automorphism of $A$. Define $e = z_1 - e'$.

By the definition of $z_0$, we get that the $A$-bimodule $(1-z_0)H$ is a sum of $A$-bimodules that are finitely generated as a right Hilbert $A$-module. It then follows from the definition of $z_1$ that we can choose a projection $e_1 \in \cZ(A)z_1$ that lies arbitrarily close to $z_1$ and for which there exists an $A$-subbimodule $L_1 \subset z_1 H$ with the following properties:
\begin{itemlist}
\item the left support of $L_1$ equals $e_1$,
\item $L_1$ is finitely generated as a right Hilbert $A$-module,
\item $\Delta^\ell_{L_1}$ is bounded and satisfies $\Delta^\ell_{L_1} \geq \delta_1 e_1$ for some real number $\delta_1 > 1$.
\end{itemlist}

Denote by $e_2$ the left support of $e_1 (H \ominus L_1)$. Making $e_1$ slightly smaller, but still arbitrarily close to $z_1$, we may assume that $e_2$ is the left support of an $A$-subbimodule $L_2 \subset e_1(H \ominus L_1)$ with the following properties: $L_2$ is finitely generated as a right Hilbert $A$-module and $\Delta^\ell_{L_2}$ is bounded. By construction, $e_2 \leq e_1$. Since $e_2 L_1$ and $L_2$ are orthogonal and have the same left support $e_2$, it follows that for nonzero projections $s \in \cZ(A)e_2$, the $A$-bimodule $s H$ is not given by a partial automorphism of $A$. This means that $e_2 \leq e$ and thus, $e_2 \leq e e_1$. Define $L = L_1 + L_2$. Using notation \eqref{eq.tK}, it follows from Lemma \ref{lem.char-bimod-partial-aut} that the left support of $e_2 L \, J(L) e_2 \cap (t_{e_2L}A)^\perp$ equals $e_2$. A fortiori, the left support of $e_2 L \, H z \cap (t_{e_2L}A)^\perp$ equals $e_2$.

We put $e_3 = e e_1 - e_2$. Since $e_2$ is the left support of $e_1(H \ominus L_1)$, we get that $e_3 H = e_3 L_1 = e_3 L$. Since $e_3 \leq e$, applying Lemma \ref{lem.char-bimod-partial-aut} to the $A$-bimodule $zH$, we conclude that the left support of $e_3 L \, Hz \cap (t_{e_3 H} A)^\perp$ equals $e_3$. Summarizing, $L$ has the following properties:
\begin{itemlist}
\item the left support of $L$ equals $e_1$,
\item $L$ is finitely generated as a right Hilbert $A$-module,
\item $\Delta^\ell_{L}$ is bounded and satisfies $\Delta^\ell_{L} \geq \delta e_1$ for some real number $\delta > 1$,
\item the left support of $L \, Hz \cap (t_L A)^\perp$ equals $e e_1$.
\end{itemlist}
Denote by $s \in \cZ(A)$ the left support of $L \, H(z_0+e_1) \cap (t_L A)^\perp$. Since $e_1$ could be chosen arbitrarily close to $z_1$, it follows that $s$ lies arbitrarily close to $e$.

We next prove that
\begin{my-enumlist}[resume]
\item\label{my-three} $\cZ(M)s \subset \cZ(A)s$,
\item\label{my-four} every $M$-central state $\omega$ on $\langle M,e_A \rangle$ that is normal on $M$ satisfies $\omega(s)=0$.
\end{my-enumlist}

Write $\Delta := \Delta^\ell_L$, choose a Pimsner-Popa basis $(\xi_i)_{i=1}^n$ for the right Hilbert $A$-module $L$ and put
$$t := t_L = \sum_{i=1}^n \xi_i \ot_A J(\xi_i) \; .$$
Since $\Delta$ is bounded, the vectors $\xi_i \in H$ are both left and right bounded.


Denoting by $P_T$ the orthogonal projection onto a Hilbert subspace $T$, the main properties of $t$, used throughout the proof, are:
$$\langle t,t\rangle_A = {}_A \langle t,t\rangle = \Delta \;\; , \;\; \ell(\xi)^* t = J(P_L(\xi)) \;\;\text{and}\;\; r(\xi)^* t = P_L(J(\xi)) \;\; ,$$
for all left and right bounded vectors $\xi \in \cH$.


Since the vectors $\xi_i$ are both left and right bounded, we can define the self-adjoint element $S_1 \in e_1 M e_1$ given by
$$S_1 := \sum_{i=1}^n W(\xi_i , J(\xi_i)) \; .$$
By Lemma \ref{lem.diffuse-element-commutant}, the von Neumann algebra $D := \{S_1\}\dpr$ is a subalgebra of $e_1 M e_1 \cap (Ae_1)'$ that is diffuse relative to $Ae_1$. We fix a unitary $u \in \cU(D)$ satisfying $E_{Ae_1}(u^k) = 0$ for all $k \in \Z \setminus \{0\}$.

Defining
$$S_k := \sum_{i_1,\ldots,i_k=1}^n W(\xi_{i_1},J(\xi_{i_1}),\ldots,\xi_{i_k}, J(\xi_{i_k})) \; ,$$
and denoting by $\Om \in L^2(M)$ the vacuum vector, we get that
\begin{equation}\label{eq.tk}
t_k := S_k \Omega = \underbrace{t \ot_A \cdots \ot_A t}_{\text{$k$ times}} \; .
\end{equation}
With the convention that $S_0 = e_1$, the elements $S_k$, $k \geq 0$ span a dense $*$-subalgebra of $D$ and are orthogonal in $L^2(D)$.


Proof of \ref{my-three}. We start by proving that an element $x \in \cZ(M)e_1$ must be of a special form. Define the von Neumann subalgebra $E \subset e_1 M e_1$ given by $E := Ae_1 \vee D$. Define $T_0 \subset H^2$ as the closure of $t A$. Note that $\ell(t)\ell(t)^*\Delta^{-1}$ is the orthogonal projection of $H^2$ onto $T_0$. Then define $T_2 := H^2 \ominus T_0$ and $T_3 := H^3 \ominus (T_0 H + H T_0)$. Observe that $L^2(e_1 M e_1 \ominus E)$ is spanned by the $D$-subbimodules
\begin{equation}\label{eq.in-the-span}
\overline{D H D} \;\; , \;\; \overline{D T_2 D} \;\; , \;\; \overline{D T_3 D} \;\; , \;\; \overline{D T_2 H^n T_2 D} \;\;\text{with}\; n \geq 0 \; .
\end{equation}
Each of the $D$-bimodules in \eqref{eq.in-the-span} is contained in a multiple of the coarse $D$-bimodule $L^2(D) \ot L^2(D)$. This is only nontrivial for the first one $\overline{D H D}$. Fix a left and right bounded vector $\mu \in H$ with $\|\mu\| \leq 1$. Using the notation $t_k$ introduced in \eqref{eq.tk}, one checks that
$$S_k W(\mu) \Om = t_k \ot_A \mu + t_{k-1} \ot_A P_L(\mu) \quad\text{and}\quad W(\mu) S_k \Om = \mu \ot_A t_k + P_{J(L)}(\mu) \ot_A t_{k-1} \; .$$
When $\mu,\eta \in H$ are left and right bounded vectors, we have $\langle t_k \ot_A \mu, \eta \ot_A t_l\rangle = 0$ if $k \neq l$, while
\begin{align*}
\langle t_k \ot_A \mu , \eta \ot_A t_k \rangle &= \langle \ell(\eta)^* (t_k \ot_A \mu), t_k \rangle \\
&= \langle J(P_L(\eta)) \ot_A t_{k-1} \ot_A \mu, t_k \rangle \\
&= \langle J(P_L(\eta)) \ot_A t_{k-1}, r(\mu)^* t_k \rangle = \langle J(P_L(\eta)) \ot_A t_{k-1}, t_{k-1} \ot P_L(J(\mu)) \rangle \; .
\end{align*}
We can continue inductively and find complex numbers $\al_k,\beta_k,\gamma_k$ with modulus at most $1$, depending on the vector $\mu$ that we keep fixed, such that
$$\langle S_k W(\mu) S_l , W(\mu) \rangle = \begin{cases} \al_k &\;\;\text{if $k=l$ and $k \geq 0$,} \\
\beta_{k-1} &\;\;\text{if $k=l+1$ and $l \geq 0$,} \\
\gamma_k &\;\;\text{if $k=l-1$ and $l \geq 1$.} \end{cases}$$
We next claim that
$$\xi := \sum_{k=0}^\infty \Bigl( \al_k (\Delta^{-k} S_k \ot \Delta^{-k} S_k)  + \be_k (\Delta^{-k-1}S_{k+1} \ot \Delta^{-k} S_k) + \gamma_k(\Delta^{-k} S_k \ot \Delta^{-k-1}S_{k+1})\Bigr)$$
is a well defined element in $L^2(E) \ot L^2(E)$. This follows because $E_A(S_k^2) = \langle t_k,t_k\rangle_A = \Delta^k$ and thus
$$\|\Delta^{-k} S_k\|_2^2 = \tau(\Delta^{-2k} S_k^2) = \tau(\Delta^{-k}) \leq \delta^{-k} \; ,$$
where $\delta > 1$. By construction,
$$\langle S_k W(\mu) S_l , W(\mu) \rangle = \tau(e_1)^{-2} \; (\tau \ot \tau)((S_k \ot S_l) \xi) \; .$$
So, the $D$-bimodule $\overline{D\mu D}$ is contained in the coarse $D$-bimodule $L^2(E) \ot L^2(E)$.

We have thus proved that all $D$-bimodules in \eqref{eq.in-the-span} are contained in a multiple of the coarse $D$-bimodule. Since $D$ is diffuse, it follows that $e_1 M e_1 \cap D' \subset E$. In particular, $\cZ(M) e_1 \subset E$.

We are now ready to prove \ref{my-three}. Fix $x \in \cZ(M)$. We have to prove that $x s \in A$. Because of \ref{my-one} and the previous paragraphs, we can uniquely decompose $x (z_0+e_1)$ as the $\|\,\cdot \,\|_2$-convergent sum
\begin{equation}\label{eq.dec-x}
x (z_0 + e_1) = a_0 + \sum_{k=1}^\infty S_k a_k
\end{equation}
with $a_0 \in A(z_0 + e_1)$ and $a_k \in A e_1$ for all $k \geq 1$. Note that $a_0 = E_A(x)(z_0+e_1)$ and $a_k = \Delta^{-k} E_A(S_k x)$ for all $k \geq 1$.

Let now $\eta \in L \, H(z_0 + e_1) \cap (t A)^\perp$ be an arbitrary left and right bounded vector. Note that
\begin{equation}\label{eq.my-eta}
\eta = \sum_{i=1}^n \xi_i \ot_A J(\eta_i)
\end{equation}
where the vectors $\eta_i \in (z_0 + e_1)H$ are both left and right bounded. Define
$$W(\eta) := \sum_{i=1}^n W(\xi_i,J(\eta_i))$$
and note that $W(\eta) \in s M (z_0 + e_1) \subset e_1 M (z_0 + e_1)$.

%

Using that $W(\eta)$ commutes with $x$ and using the decomposition of $x(z_0 + e_1)$ in \eqref{eq.dec-x}, we find that
\begin{align*}
W(\eta) x \Omega &= W(\eta) (z_0 + e_1) x \Omega = W(\eta) a_0 \Omega + \sum_{k=1}^\infty W(\eta) S_k a_k \Omega \\
&= \eta (a_0 + a_1) + \sum_{k=1}^\infty \eta \ot_A t_k(a_k + a_{k+1}) \;\; ,\\
x W(\eta) \Omega &= x e_1 W(\eta) \Omega = a_0 e_1 W(\eta) \Omega + \sum_{k=1}^\infty a_k S_k W(\eta) \Omega \\
&= (a_0 + a_1) \eta + \sum_{k=1}^\infty (a_k+a_{k+1}) t_k \ot_A \eta \;\; .
\end{align*}
In this last expression for $x W(\eta) \Omega$, all terms except $(a_0 + a_1)\eta$ are orthogonal to $W(\eta) x \Omega$. We conclude that $(a_k+a_{k+1}) t_k \ot_A \eta = 0$ for all $k \geq 1$ and for all choices of $\eta$. Since the left support of $L \, H(z_0 + e_1) \cap (t A)^\perp$ equals $s$, it follows that $(a_k + a_{k+1})s = 0$ for all $k \geq 1$. This means that $a_k s = (-1)^{k-1} a_1 s$ for all $k \geq 1$.

Since,
$$+\infty > \|x\|_2^2 \geq \sum_{k=1}^\infty \|S_k a_k s\|_2^2 = \sum_{k=1}^\infty \tau(s a_1^* \Delta^k a_1 s) \geq \sum_{k=1}^\infty \delta^k \|a_1 s\|_2^2 \; ,$$
it follows that $a_1 s = 0$. So, $a_k s = 0$ for all $k \geq 1$. From \eqref{eq.dec-x}, it follows that $x s \in A$, so that \ref{my-three} is proved.

Proof of \ref{my-four}. Fix an $M$-central state $\om$ on $\langle M,e_A \rangle$ that is normal on $M$. We have to prove that $\om(s) = 0$. Recall that we defined $T_0 \subset H^2$ as the closure of $t A$. Consider the following orthogonal decomposition of $e_1 L^2(M)$ as an $A$-bimodule:
\begin{align*}
& e_1 L^2(M) = V_0 \oplus V_1 \oplus V_2 \quad\text{where} \quad V_0 := \bigoplus_{n=0}^\infty T_0 H^n \;\; ,\\
& V_1 := L^2(A e_1) \oplus \bigoplus_{n=0}^\infty (e_1 H \ominus L)H^n \;\; , \quad V_2 := L \oplus \bigoplus_{n=0}^\infty (LH \ominus T_0)H^n \;\; .
\end{align*}
Denote by $Q_i \in e_1 \langle M,e_A \rangle e_1$ the projections onto $V_i$, for $i=0,1,2$. So, $e_1 = Q_0+Q_1+Q_2$. Also note that the projections $Q_i$ commute with $A$. We prove below that $\om(s Q_0) = \om(Q_1)=\om(Q_2) = 0$. Once these statements are proved, \ref{my-four} follows.

%

To prove that $\om(Q_1) = 0$, note that for all $\mu \in V_1$ and all $k \geq 1$, we have that $S_k \mu = t_k \ot_A \mu$ and thus, $S_k \mu$ is orthogonal to $V_1$. So, for all $\mu,\mu' \in V_1$ and $d \in D$, we get that
$$\langle d \mu, \mu' \rangle = \tau(e_1)^{-1} \tau(d) \, \langle \mu , \mu' \rangle \; .$$
Above we introduced the unitary element $u \in \cU(D)$ satisfying $\tau(u^k)=0$ for all $k \in \Z \setminus \{0\}$. It follows that the subspaces $u^k V_1$ are all orthogonal. So, the projections $u^k Q_1 u^{-k}$ are all orthogonal. By $M$-centrality, $\om$ takes the same value on each of these projections. So, $\om(Q_1) = 0$.

To prove that $\om(Q_2) = 0$, we argue similarly. For all $\mu \in V_2$ and all $k \geq 2$, we have that $S_k \mu = t_k \ot_A \mu + t_{k-1} \ot_A \mu$ and thus, $S_k \mu$ is orthogonal to $V_2$. On the other hand, $S_1 \mu = t \ot_A \mu + \mu$ and here, only $t \ot_A \mu$ is orthogonal to $V_2$. It follows that for all $\mu,\mu' \in V_2$ and $d \in D$,
$$\langle d \mu,\mu' \rangle = \gamma(d) \, \langle \mu,\mu' \rangle \; ,$$
where $\gamma : D \recht \C$ is the normal state given by $\gamma(e_1) = \gamma(S_1) = 1$ and $\gamma(S_k) = 0$ for all $k \geq 2$. Note that $\gamma$ can be defined as well as the vector state on $D$ implemented by any choice of unit vector in $V_2$. Since $D$ is diffuse, we can choose a unitary $v \in \cU(D)$ such that $\gamma(v^k) = 0$ for all $k \in \Z \setminus \{0\}$. It follows that the subspaces $v^k V_2$ are all orthogonal. As in the previous paragraph, we get that $\om(Q_2) = 0$.
%
%
%

It remains to prove that $\om(s Q_0) = 0$. Fix $\eta \in L \, H(z_0 + e_1) \ominus T_0$ as in \eqref{eq.my-eta} and define
$$\eta' = \sum_{i=1}^n \eta_i \ot_A J(\xi_i) \; .$$
Note that $\eta' \in (z_0 + e_1) H \, J(L) \ominus T_0$. From \ref{my-two}, we already know that $\om(z_0)= 0$. Since $e_1 \eta' \in V_1 + V_2$, we also know that $\om(\ell(e_1 \eta')\ell(e_1 \eta')^*) = 0$. Both together imply that $\om(\ell(\eta') \ell(\eta')^*) = 0$.

For all $n \geq 0$ and $\mu \in H^n$, we have that
$$W(\eta) (\eta' \ot_A t \ot_A \mu) = \eta \ot_A \eta' \ot_A t \ot_A \mu + \sum_{i=1}^n \ell(\xi_i) \ell(\eta_i)^* (\eta' \ot_A t \ot_A \mu) + \langle \eta',\eta' \rangle_A \, (t \ot_A \mu) \; .$$
Since
$$\ell(t)^* \sum_{i=1}^n \ell(\xi_i) \ell(\eta_i)^* \eta' = \sum_{i=1}^n \ell(J(\xi_i))^* \ell(\eta_i)^* \eta' = \ell(\eta')^* \eta' = \langle \eta',\eta'\rangle_A$$
and since the projection $Q_0$ is given by $Q_0 = \Delta^{-1} \ell(t)  \ell(t)^*$, we get that
$$Q_0 W(\eta) (\eta' \ot_A t \ot_A \mu) = \langle \eta',\eta' \rangle_A \, \Delta^{-1} \, (t \ot_A t \ot_A \mu) + \langle \eta',\eta' \rangle_A \, (t \ot_A \mu)$$
for all $n \geq 0$ and all $\mu \in H^n$. This means that
$$Q_0 W(\eta) \ell(\eta' \ot_A t) = \langle \eta',\eta' \rangle_A \, \bigl( \Delta^{-1} \, \ell(t \ot_A t) + \ell(t)\bigr) = \ell(t) \, \langle \eta',\eta' \rangle_A \, (1 + \Delta^{-1} \ell(t)) \; .$$
Because
$$\|\Delta^{-1} \ell(t)\|^2 = \|\Delta^{-2} \ell(t)^* \ell(t)\| = \|\Delta^{-1}\| \leq \delta^{-1} < 1 \; ,$$
the operator $R := 1 + \Delta^{-1} \ell(t)$ is invertible. Also note that there exists a $\kappa > 0$ such that
$$\ell(\eta' \ot_A t) \ell(\eta' \ot_A t)^*  \leq \kappa \, \ell(\eta') \ell(\eta')^* \; .$$
So, we find $\eps > 0$ and $\kappa > 0$ such that
\begin{equation}\label{eq.imp-ineq}
\begin{split}
\eps \; \ell(t) \, (\langle \eta',\eta' \rangle_A )^2 \, \ell(t)^* & \leq \ell(t) \, \langle \eta',\eta' \rangle_A \, RR^* \, \langle \eta',\eta' \rangle_A \, \ell(t)^* \\
&= Q_0 W(\eta) \ell(\eta' \ot_A t) \ell(\eta' \ot_A t)^* W(\eta)^* Q_0 \\
& \leq \kappa Q_0 W(\eta) \ell(\eta') \ell(\eta')^* W(\eta)^* Q_0 \; .
\end{split}
\end{equation}
We already proved that $\om(\ell(\eta') \ell(\eta')^*) = 0$. Since $\om$ is $M$-central, also
$$\om(W(\eta) \ell(\eta') \ell(\eta')^* W(\eta)^*)=0 \; .$$
Because $e_1 = Q_0 + Q_1 + Q_2$ and $\om(Q_1) = \om(Q_2) = 0$, the Cauchy-Schwarz inequality implies that $\om(Y) = \om(Q_0 Y) = \om(Y Q_0)$ for all $Y \in e_1 \langle M,e_A \rangle e_1$. Therefore,
$$\om\bigl( Q_0 W(\eta) \ell(\eta') \ell(\eta')^* W(\eta)^* Q_0 \bigr) = \om(W(\eta) \ell(\eta') \ell(\eta')^* W(\eta)^*) = 0 \; .$$
It then follows from \eqref{eq.imp-ineq} that
$$\om\bigl( (\langle \eta',\eta' \rangle_A )^2 \, \Delta \, Q_0\bigr) = 0$$
for all bounded vectors $\eta' \in (z_0 + e_1) H \, J(L) \ominus T_0$. By the Cauchy-Schwarz inequality and the normality of $\om$ restricted to $M$, we get that $\om(a_i Q_0) \recht \om(a Q_0)$ whenever $a_i \in A$ is a bounded sequence such that $\|a_i-a\|_2 \recht 0$. Since the right support of the $A$-bimodule $(z_0 + e_1) H \, J(L) \ominus T_0$ equals $s$, it follows that $\om(s Q_0) = 0$. Since we already proved that $\om(Q_1) = \om(Q_2) = 0$, it follows that \ref{my-four} holds.

Since $s$ lies arbitrarily close to $e$, it follows from \ref{my-one}-\ref{my-two} and \ref{my-three}-\ref{my-four} that

\begin{my-enumlist}[resume]
\item\label{my-five} $\cZ(M)(z_0+e) \subset \cZ(A)(z_0+e)$,
\item\label{my-six} every $M$-central state $\omega$ on $\langle M,e_A \rangle$ that is normal on $M$ satisfies $\omega(z_0+e)=0$.
\end{my-enumlist}

Recall that $z=z_0+z_1$ and $z_2 = 1-(z_0+z_1)$. Note that $\Delta^\ell_{z_2 H} \leq z_2$. We claim that $z_2 H z_2 = \{0\}$. Denote by $e_0 \in \cZ(A)z_2$ the left support of $z_2 H z_2$. Note that by symmetry, $e_0$ also is the right support of $z_2 H z_2$. By Lemma \ref{lem.again-partial-aut}, we get that $\Delta^\ell_{e_0 H e_0} = e_0$ and that $e_0 H e_0$ is given by a partial automorphism of $A$. Since
$$\Delta^\ell_{e_0H} = \Delta^\ell_{e_0He_0} + \Delta^\ell_{e_0H(1-e_0)} = e_0 + \Delta^\ell_{e_0H(1-e_0)}$$
and since $\Delta^\ell_{e_0H} \leq e_0$, we get that $e_0H(1-e_0) = \{0\}$. We conclude that $e_0H = He_0 = e_0He_0$ and that this $A$-bimodule is given by a partial automorphism of $A$. Since $H$ is assumed to be completely nontrivial, we get that $e_0=0$ and the claim is proved.

Recall that $e \in \cZ(A)z_1$ was defined as $e = z_1 - e'$ where $e' \in \cZ(A)z_1$ has the following properties: denoting by $f \in \cZ(A)$ the right support of $e'H$, we have that $e'H = z H f$ and that the $A$-bimodule $e'H$ is given by a partial automorphism of $A$. We claim that $f \leq z$. To prove this claim, denote $f_1 := f z_2$. If $f_1 \neq 0$, we find a nonzero projection $e\dpr \in \cZ(A)e'$ such that $e\dpr H = z H f_1$ and such that this $A$-bimodule is given by a partial automorphism of $A$. Above, we have proved that $z_2 H z_2 = \{0\}$. A fortiori, $z_2 H f_1 = \{0\}$, meaning that $H f_1 = z H f_1$. But then, $e\dpr H = H f_1$, contradicting the complete nontriviality of $H$. So, we have proved that $f \leq z$.

We next claim that $f \leq z_0+e$. To prove this claim, assume that $f' := f e'$ is nonzero. Then, $f'H = f e'H = f z H f \subset Hz$ because $f \leq z$. Applying the symmetry $J$, it follows that $H f' = z H f'$ and thus $e\dpr H = H f'$ for some nonzero projection $e\dpr \in \cZ(A)e'$, again contradicting the complete nontriviality of $H$. So, we have proved that $f \leq z_0+e$.

Since $e'H$ is given by a partial automorphism of $A$, we can take projections $e\dpr \in \cZ(A)e'$ arbitrarily close to $e'$ such that $e\dpr H$ is finitely generated as a right Hilbert $A$-module and $\Delta^\ell_{e\dpr H}$ is bounded. Denote by $f' \in \cZ(A)f$ the right support of $e\dpr H$ and denote by $\al : \cZ(A)e\dpr \recht \cZ(A)f'$ the corresponding surjective $*$-isomorphism satisfying $a \xi = \xi \al(a)$ for all $a \in \cZ(A)e\dpr$. Let $(\gamma_i)_{i=1}^n$ be a Pimsner-Popa basis of the right $A$-module $e\dpr H$ and define
$$R_i = \ell(\gamma_i) +\ell(J(\gamma_i))^* \quad\text{and}\quad R = \sum_{i=1}^n R_i R_i^* = \Delta^\ell_{e\dpr H} + \sum_{i=1}^n W(\gamma_i,J(\gamma_i)) \; .$$
Note that $R_i \in e\dpr M f'$ and $R \in e\dpr M e\dpr$. Since $\Delta^\ell_{e\dpr H} = e\dpr \Delta^\ell_H \geq e \dpr$, it follows from Lemma \ref{lem.diffuse-element-commutant} that the support projection of $R$ equals $e\dpr$.

Let $x \in \cZ(M)$ and using \ref{my-five}, take $a \in \cZ(A)(z_0+e)$ such that $(z_0+e)x = a$. Since $f'\leq z_0+e$, we have $f' x = a f'$ and thus
$$x R = \sum_{i=1}^n R_i x R_i^* = \sum_{i=1}^n R_i \, af' \, R_i^* = \al^{-1}(af') R \; .$$
Since the support projection of $R$ equals $e\dpr$, we have proved that $\cZ(M) e\dpr \subset \cZ(A) e\dpr$. Since $e\dpr$ lies arbitrarily close to $e'$, together with \ref{my-five}, it follows that

\begin{my-enumlist}[resume]
\item\label{my-seven} $\cZ(M) z \subset \cZ(A)z$.
\end{my-enumlist}

A similar reasoning using \ref{my-six} then implies that

\begin{my-enumlist}[resume]
\item\label{my-eight} every $M$-central state $\omega$ on $\langle M,e_A \rangle$ that is normal on $M$ satisfies $\omega(z)=0$.
\end{my-enumlist}

To prove the first two statements of the theorem, it remains to see what happens under the projection $z_2$.

Denote $\Delta_2 := \Delta^\ell_{z_2 H}$. By the definition of $z_2$, we have that $\Delta_2 \leq z_2$. Let $(\mu_i)_{i \in I}$ be a (possibly infinite) Pimsner-Popa basis for the right $A$-module $z_2 H$. Since $\Delta_2$ is bounded, we may choose the vectors $\mu_i$ to be left and right bounded. For the same reason,
$$s := \sum_{i \in I} \mu_i \ot_A J(\mu_i)$$
is a well defined bounded $A$-central vector in $z_2 H \, H z_2$ and the infinite sums
$$G_n = \sum_{i_1,\ldots,i_n} W(\mu_{i_1},J(\mu_{i_1}),\ldots,\mu_{i_n},J(\mu_{i_n}))$$
are well defined bounded operators in $z_2 M z_2 \cap (Az_2)'$ satisfying
$$G_n \Omega = s_n := \underbrace{s \ot_A \cdots \ot_A s}_{\text{$n$ times}} \; .$$
By convention, we put $G_0 = z_2$. From the definition of $G_n$, we obtain the recurrence relation
\begin{equation}\label{eq.recurrence-Gn}
G_1 G_n = G_{n+1} + G_n + \Delta_2 \, G_{n-1}
\end{equation}
for all $n \geq 1$, and thus, $G_{n+1} = (G_1-1)G_n - \Delta_2 \, G_{n-1}$ for all $n \geq 1$.

Denote by $q \in z_2 M z_2$ the projection onto the kernel of $G_1 + \Delta_2$. Although the sum defining $G_1$ is infinite, the computations in the proof of Lemma \ref{lem.diffuse-element-commutant} remain valid and it follows that the kernel of $(G_1 + \Delta_2) \, 1_{\{1\}}(\Delta_2)$ is reduced to zero. So, $q \leq 1_{(0,1)}(\Delta_2)$.

With the convention that $s_0 = z_2 \Omega$, we claim that
\begin{equation}\label{eq.formula-q}
q \Omega = \sum_{k=0}^\infty (-1)^k (z_2 - \Delta_2) s_k = \sum_{k=0}^\infty (-1)^k  s_k (z_2 - \Delta_2) \; .
\end{equation}
Because
\begin{align*}
\sum_{k=0}^\infty \| (z_2-\Delta_2) s_k\|_2^2 &= \sum_{k=0}^\infty \tau\bigl( \langle s_k,s_k\rangle_A \, (z_2-\Delta_2)^2 \bigr) \\
&= \sum_{k=0}^\infty \tau\bigl( \Delta_2^k (z_2 - \Delta_2)^2 \bigr) = \tau(z_2 -\Delta_2) < \infty \; ,
\end{align*}
the right hand side of \eqref{eq.formula-q} is a well defined element $p \in L^2(z_2 M z_2)$ satisfying, with $\|\,\cdot\,\|_2$-convergence,
$$p = \sum_{k=0}^\infty (-1)^k (z_2 - \Delta_2) \, G_k \; .$$
Note that $p=p^*$. Using the recurrence relation \eqref{eq.recurrence-Gn}, it follows that $(G_1 + \Delta_2) p = 0$ and thus $p = qp$. Taking the adjoint, also $p = pq$.

On the other hand, because $(G_1+\Delta_2)q = 0$, we have $G_1 q = - \Delta_2 q$. Using the recurrence relation \eqref{eq.recurrence-Gn}, it follows that $G_k q = (-1)^k \Delta_2^k q$ for all $k \geq 0$. It then follows that
$$p q = \sum_{k=0}^\infty (z_2- \Delta_2) \Delta_2^k \, q = 1_{(0,1)}(\Delta_2) \, q = q \; .$$
We already proved that $p q = p$, so that $p = q$ and \eqref{eq.formula-q} is proved.

From \eqref{eq.formula-q}, we get for all $\xi \in \cH$ that
$$(\ell(\xi) + \ell(J(\xi))^*) \, q \, \Omega = (\ell(\xi z_2) + \ell(J(\xi z_2))^*) \, q \, \Omega = 0 \; .$$
So, for all $x \in M$, we have that $x q = E_A(x) q$. Taking the adjoint, also $q x = q E_A(x)$ for all $x \in M$. Since $q$ commutes with $A$, it follows that $q \in \cZ(M)$ and $M q = A q$. From \eqref{eq.formula-q}, we also get that $E_A(q) = z_2 - \Delta_2$ and thus $E_A(q) = Z(\Delta^\ell_H)$ where $Z : (0,+\infty) \recht \R$ is defined as in the formulation of the theorem. So, $E_A(1-q) = z + \Delta_2$ and this operator has support equal to $1$. Statement \ref{stat-c} of the theorem is now proven.

We next prove that
\begin{my-enumlist}[resume]
\item\label{my-nine} $\cZ(M) (z_2-q) \subset \cZ(A)(z_2 - q)$.
\end{my-enumlist}

Take $x \in \cZ(M)$ and write
$$xz_2 \Om = \sum_{n =0}^\infty \zeta_n \quad\text{with}\;\; \zeta_n \in z_2 H^n \; .$$
Using \ref{my-seven}, take $a \in \cZ(A) z$ such that $x z = a$. Also write $a_0 = E_A(xz_2)$ and note that $\zeta_0 = a_0 \Om$.

Since $z_2 \cH z_2 = 0$, we have $z_2 \cH = z_2 \cH z$ and we get, for every $\xi \in \cH$, that
\begin{align*}
\sum_{n=0}^\infty \bigl( \ell(\xi)^* + \ell(J(\xi)) \bigr) \, \zeta_n & = \bigl( \ell(\xi)^* + \ell(J(\xi)) \bigr) \, x z_2 \Om \\
&= x \, \bigl( \ell(\xi)^* + \ell(J(\xi)) \bigr) \, z_2 \Om = x \, J(z_2 \xi) = xz \, J(z_2 \xi) = a \, J(z_2 \xi) \; .
\end{align*}
Comparing the components in $H^n$ for all $n \geq 0$, we find that
$$\ell(\xi)^* \zeta_1 = 0 \quad , \quad  \ell(\xi)^* \zeta_2 = a \, J(\xi) - J(\xi) \, a_0 \quad , \quad \ell(\xi)^* \zeta_{n+1} = - J(\xi) \ot_A \zeta_{n-1}$$
for all $\xi \in z_2 \cH$ and all $n \geq 2$. Since $\zeta_n \in z_2 H^n$ for all $n$, it first follows that $\zeta_1 = 0$ and then inductively, that $\zeta_n = 0$ for all odd $n$.

Next, we get that $\zeta_2 = s_a - s a_0$, where
$$s_a := \sum_{i \in I} \mu_i \ot_A a J(\mu_i)$$
is a well defined $A$-central vector in $z_2 H^2 z_2$.

Before continuing the proof, we give another expression for $s_a$. For all $\mu,\mu' \in z_2 \cH = z_2 \cH z$, we have that $W(J(\mu),\mu') \in z Mz$. Since $x z = a$ and $x \in \cZ(M)$, it follows that $a$ commutes with $W(J(\mu),\mu')$. This means that
$$a \, J(\mu) \ot_A \mu' = J(\mu) \ot_A \mu' \, a \quad\text{for all}\;\; \mu,\mu' \in z_2 \cH \; .$$
It follows that $a \, J(\mu) \ot_A s = J(\mu) \ot_A s_a$ for all $\mu \in z_2 \cH$. Defining the normal completely positive map $\vphi : Az \recht Az_2$ given by
$$\vphi(b) = \sum_{i \in I} \langle J(\mu_i) , b \, J(\mu_i)\rangle_A \quad\text{for all}\;\; b \in Az \; ,$$
we get that $\vphi(a) \, s = \Delta_2 \, s_a$. Since $\vphi(z) = \Delta_2$, there is a unique normal completely positive map $\psi : Az \recht Az_2$ such that $\psi(b) \Delta_2 = \vphi(b)$ for all $b \in Az$. We conclude that $s_a = \psi(a) \, s = s \, \psi(a)$.

Writing $a_1 = \psi(a) - a_0$, we get that $\zeta_2 = s \, a_1$. We then conclude that $\zeta_{2n} = (-1)^{n+1}\, s_n \, a_1$ for all $n \geq 1$. Define the spectral projection $r = 1_{\{1\}}(\Delta_2)$. Since
$$\langle \zeta_{2n},\zeta_{2n}\rangle_A = a_1^* \, \langle s_n,s_n\rangle_A \, a_1 = a_1^* \, \Delta_2^n \, a_1 \; ,$$
we get that $\|\zeta_{2n} r\| = \|a_1 r\|_2$ for all $n$. Since $\sum_n \|\zeta_{2n} r\|^2 < \infty$, we conclude that $a_1 r = 0$ and thus $x r \in A$.

Using \eqref{eq.formula-q}, it follows that $x (z_2 - \Delta_2) = q a_1 + a_2$ for some element $a_2 \in A$. Since $xr \in A$, it follows that $x(z_2-q) \in A(z_2-q)$. Since the support of $E_A(z_2-q)$ equals $z_2$, it follows that \ref{my-nine} holds.

Using \ref{my-seven} and \ref{my-nine}, to conclude the proof of statement \ref{stat-d}, it suffices to prove that for any $a \in \cZ(A)$, we have $a(1-q) \in \cZ(M)$ if and only if $a \in C$, where $C$ is defined in the formulation of the theorem. This follows immediately by expressing the commutation with $\ell(\xi)+\ell(J(\xi))^*$ for all $\xi \in \cH$ and using that $\bigl(\ell(\xi)+\ell(J(\xi))^*\bigr) \, q = 0$, as shown above.

Let $\om$ be an $M$-central state on $\langle M,e_A \rangle$ that is normal on $M$. To conclude the proof of statement \ref{stat-a}, we have to show that $\om(1-q) = 0$. By \ref{my-eight}, we already know that $\om(z) = 0$. With $\mu_i \in z_2 \cH = z_2 \cH z$ as above, define $y_i := \ell(\mu_i) + \ell(J(\mu_i))^*$. Note that $y_i \in z_2 M z$ and that $G_1 + \Delta_2 = \sum_i y_i y_i^*$. By $M$-centrality and normality of $\om$ on $M$, and because $y_i^* y_i \in z M z$, we get that $\om(G_1 + \Delta_2) = 0$. So, $\om(z_2 - q)=0$. Since we already know that $\om(z) = 0$, we conclude that $\om(1-q) = 0$.


It remains to prove statement \ref{stat-b}. Assume that $s \in \cZ(M)(1-q)$ is a nonzero projection and that $B \subset Ms$ is a Cartan subalgebra. Since $\cN_{Ms}(B)\dpr = Ms$, a combination of statement \ref{stat-a} and Theorem \ref{thm1} implies that $B \prec_{M} A(1-q)$. The $A$-subbimodule $z_2 H = z_2 H z$ of $L^2(M)$ has finite right $A$-dimension equal to $\tau(\Delta_2)$ and realizes a full intertwining of $A(z_2-q)$ into $Az$. It then follows that $B \prec_M A z$.

By \cite[Theorem 2.1]{Po03}, we can take projections $q_1 \in B$, $p \in A z$, a faithful normal unital $*$-homomorphism $\theta : Bq_1 \recht pAp$ and a nonzero partial isometry $v \in q_1 M p$ such that $b v = v \theta(b)$ for all $b \in Bq_1$. Since $B \subset Ms$ is maximal abelian, we may assume that $v v^* = q_1$. By \cite[Lemma 1.5]{Io11}, we may assume that $B_0 := \theta(Bq_1)$ is a maximal abelian subalgebra of $pAp$. Write $q_2 = v^* v$ and note that $q_2 \in B_0' \cap pMp$. We may assume that the support projection of $E_A(q_2)$ equals $p$.

Since $z = z_0 + z_1$, at least one of the projections $p z_0$, $p z_1$ is nonzero. Since we can cut down everything with the projections $z_0$ and $z_1$, we may assume that either $p \leq z_0$ or $p \leq z_1$.

{\bf Proof in the case where $p \leq z_0$.} Recall that we denoted by $K \subset H$ the largest $A$-subbimodule that is left weakly mixing and that $z_0$ is the left support of $K$. First assume that the $B_0$-$A$-bimodule $p K$ is left weakly mixing. Define the orthogonal decomposition of the $pAp$-bimodule $pL^2(M) p$ given by
$$p L^2(M) p = U_1 \oplus U_2 \quad\text{with}\quad U_1 = \bigoplus_{n=0}^\infty p K H^n p \quad\text{and}\quad U_2 = L^2(pAp) \oplus \bigoplus_{n=0}^\infty p(H \ominus K) H^n p \; .$$
We claim that $v^* \cN_{q_1Mq_1}(Bq_1) v \subset U_2$. To prove this claim, take $u \in \cN_{q_1Mq_1}(Bq_1)$ and write $u^* b u = \al(b)$ for all $b \in Bq_1$. Put $x = v^* u v$ and denote by $y$ the orthogonal projection of $x$ onto $U_1$. Since $U_1$ is a $pAp$-subbimodule of $pL^2(M)p$, we get that $y$ is a right $pAp$-bounded vector in $U_1$ and that $\theta(b)y = y \theta(\al(b))$ for all $b \in Bq$. Since the $B_0$-$A$-bimodule $p K$ is left weakly mixing, also $U_1$ is left weakly mixing as a $B_0$-$pAp$-bimodule. So, we can take a sequence of unitaries $b_n \in \cU(Bq_1)$ such that $\lim_n \| \langle \theta(b_n) y,y \rangle_{pAp} \|_2 = 0$. But,
$$\langle \theta(b_n) y,y \rangle_{pAp} = \langle y \theta(\al(b_n)),y \rangle_{pAp} = \theta(\al(b_n)^*) \, \langle y , y \rangle_{pAp} \; .$$
Since $\theta(\al(b_n))$ is a unitary in $B_0$, we have $\|\theta(\al(b_n)^*) \, \langle y , y \rangle_{pAp}\|_2 = \|\langle y,y \rangle_{pAp}\|_2$ for all $n$. We conclude that $y = 0$ and thus $v^* u v \in U_2$. Since the linear span of $\cN_{q_1Mq_1}(Bq_1)$ is $\|\cdot\|_2$-dense in $q_1Mq_1$, we get that $q_2 M q_2 \subset U_2$.

Again consider the von Neumann subalgebra $N \subset z_0 M z_0$ introduced in \eqref{eq.vnalg-N}. Since
$$P_{pL^2(N)p}(U_2) \subset L^2(pAp) \; ,$$
we get that $E_{pNp}(q_2 M q_2) \subset p A p$. Denote by $N_0 \subset pNp$ the von Neumann algebra generated by the subspace $E_{pNp}(q_2 M q_2)$. So, $N_0 \subset pAp$. In particular, $E_N(q_2) \in A$, so that $E_N(q_2) = E_A(q_2)$ and thus, $E_N(q_2)$ has support $p$. By \cite[Lemma 1.6]{Io11}, the inclusion $N_0 \subset pNp$ is essentially of finite index in the sense of Definition \ref{def.ess-finite-index}. A fortiori, $pAp \subset pNp$ is essentially of finite index. This contradicts the left weak mixing of the $N$-$A$-bimodule $L^2(N)$ that we obtained in \ref{NA-weak-mix}.

Next assume that the $B_0$-$A$-bimodule $p K$ is not left weakly mixing and take a nonzero $B_0$-$A$-subbimodule $K_1 \subset p K$ that is finitely generated as a right Hilbert $A$-module. Denote by $z_0' \in \cZ(B_0)$ the support projection of the left action of $B_0$ on $K_1$. Since $K_1 \neq \{0\}$, also $z_0' \neq 0$. Since the support of $E_A(q_2)$ equals $p$, we get that $E_A(q_2 z_0') = E_A(q_2) z_0' \neq 0$. So, $q_2 z_0' \neq 0$ and we can cut down everything by $z_0'$ and assume that the left $B_0$ action on $K_1$ is faithful.

Put $P = \cN_{pAp}(B_0)\dpr$. Whenever $u \in \cN_{q_1Mq_1}(Bq_1)$ with $u b u^* = \al(b)$ for all $b \in Bq_1$, we have $E_A(v^* u v) \theta(b) = \theta(\al(b)) E_A(v^* u v)$ for all $b \in Bq_1$. Since $B_0 \subset pAp$ is maximal abelian, it follows that $E_A(v^* u v) \in P$. So $E_A(q_2 M q_2) \subset P$. From \cite[Lemma 1.6]{Io11}, we conclude that the inclusion $P \subset pAp$ is essentially of finite index in the sense of Definition \ref{def.ess-finite-index}. So, all conditions of Lemma \ref{lem.complement} are satisfied and we can choose a diffuse abelian von Neumann subalgebra $D \subset B_0' \cap pMp$ that is in tensor product position w.r.t.\ $B_0$. Since $Bq_1 \subset q_1Mq_1$ is maximal abelian, also $B_0 q_2 \subset q_2 M q_2$ is maximal abelian. So, $q_2 (B_0' \cap pMp) q_2 = B_0 q_2$, contradicting Lemma \ref{lem.no-complement} below.

{\bf Proof in the case where $p \leq z_1$.} As proven above, we can find projections $e_1 \in \cZ(A)z_1$ that lie arbitrarily close to $z_1$ and for which there exists an $A$-subbimodule $L \subset z_1 H$ with the following properties: the left support of $L$ equals $e_1$, $L$ is finitely generated as a right Hilbert $A$-module, $\Delta^\ell_L$ is bounded and $\Delta^\ell_L \geq e_1$. Taking $e_1$ close enough to $z_1$ and cutting down with $e_1$, we may assume that $p \leq e_1$. By Lemma \ref{lem.diffuse-element-commutant}, we can choose a diffuse abelian von Neumann subalgebra $D \subset (Ae_1)' \cap e_1 M e_1$ that is in tensor product position w.r.t.\ $Ae_1$. Then $Dp \subset B_0' \cap pMp$ and $Dp$ is in tensor product position w.r.t.\ $B_0$. Since $Dp$ is diffuse abelian and $q_2 \in B_0' \cap pMp$ is a projection satisfying $q_2(B_0' \cap pMp)q_2 = B_0 q_2$, this again contradicts Lemma \ref{lem.no-complement}.
\end{proof}

In the proof of Theorem \ref{thm.absence-Cartan}, we needed several technical lemmas that we prove now.

Let $(A,\tau)$ be a tracial von Neumann algebra and denote by $\widehat{\cZ(A)}$ the \emph{extended positive part} of $\cZ(A)$, i.e.\ when we identify $\cZ(A) = L^\infty(X,\mu)$, then $\widehat{\cZ(A)}$ consists of all measurable functions $f : X \recht [0,+\infty]$ up to identification of functions that are equal almost everywhere.

Whenever $(B,\tau)$ and $(A,\tau)$ are tracial von Neumann algebras and $H$ is a $B$-$A$-bimodule, we denote by $\Delta^\ell_H \in \widehat{\cZ(B)}$ the unique element in the extended positive part of $\cZ(B)$ characterized by
\begin{equation}\label{eq.def-Delta}
\tau(\Delta^\ell_H e) = \dim_{-A}(e H) \quad\text{for all projections}\;\; e \in \cZ(B) \; .
\end{equation}
Writing $H \cong p(\ell^2(\N) \ot L^2(A))$ with the bimodule action given by $b \cdot \xi \cdot a = \al(b) \xi a$ where $\al : B \recht p (B(\ell^2(\N)) \ovt A) p$ is a normal $*$-homomorphism, we get that $\tau(\Delta^\ell_H \, \cdot \,) = (\Tr \ot \tau)\al(\, \cdot \,)$ and this also allows to construct $\Delta^\ell_H$.

Recall that a \emph{finitely generated} right Hilbert $A$-module $K$ admits a \emph{Pimsner-Popa basis}, i.e.\ right bounded elements $\xi_1,\ldots,\xi_n$ such that
\begin{equation}\label{eq.PP-basis}
\xi = \sum_{i=1}^n \xi_i \, \langle \xi_i,\xi\rangle_A
\end{equation}
for all right bounded elements $\xi \in K$. We denote by $t_K \in K \ot_A \overline{K}$ the associated vector given by
\begin{equation}\label{eq.tK}
t_K := \sum_{i=1}^n \xi_i \ot_A \overline{\xi_i} \; .
\end{equation}
When $K$ is an $A$-bimodule, then $t_K$ is an $A$-central vector and $\langle t_K , t_K \rangle_A = \Delta^\ell_K$.

Recall from the beginning of this section the notion of an $A$-bimodule given by a partial automorphism of $A$. Given an $A$-bimodule $L$, denote by $\zdim_{-A}(L)$, resp.\ $\zdim_{A-}(L)$, the center valued dimension of $L$ as a right, resp.\ left $A$-module. These are elements in the extended positive part of $\cZ(A)$. We have that $L$ is finitely generated as a right Hilbert $A$-module if and only if $\zdim_{-A}(L)$ is bounded.

\begin{lemma}\label{lem.char-bimod-partial-aut}
Let $(A,\tau)$ be a tracial von Neumann algebra and $T$ an $A$-bimodule with left support $e$. Denote $\Sigma := \zdim_{A-}(T \ot_A \overline{T})$. Then, the support of $\Sigma$ equals $e$ and $\Sigma \geq e$. Defining $e_1 = 1_{\{1\}}(\Sigma)$, the following holds.
\begin{enumlist}
\item Denoting by $f_1 \in \cZ(A)$ the right support of $e_1 T$, we have that $e_1 T = T f_1$ and that the $A$-bimodule $e_1 T$ is given by a partial automorphism of $A$.
\item When $e_2 \in \cZ(A)e$ and $f_2 \in \cZ(A)$ are projections such that $e_2 T = T f_2$ and such that the $A$-bimodule $e_2 T$ is given by a partial automorphism of $A$, then $e_2 \leq e_1$.
\item If $e_0 \in \cZ(A) e$ is a projection such that $e_0 T$ is finitely generated as a right Hilbert $A$-module, then the left support of $e_0 T \ot_A \overline{T} \cap (t_{e_0T} A)^\perp$ equals $e_0(1-e_1)$.
\end{enumlist}
\end{lemma}
\begin{proof}
Choose a set $I$, a projection $p \in B(\ell^2(I)) \ovt A$ and a normal unital $*$-homomorphism $\al : A \recht p(B(\ell^2(I)) \ovt A)p$ such that $T \cong p (\ell^2(I) \ot L^2(A))$ with the $A$-bimodule structure given by $a \cdot \xi \cdot b = \al(a) \xi b$. Note that $e$ equals the support of $\al$. Also note that $T \ot_A \overline{T} \cong L^2(p(B(\ell^2(I)) \ovt A)p)$ with the $A$-bimodule structure given by $a \cdot \xi \cdot b = \al(a) \xi \al(b)$.

Define $e_0 = 1_{(0,1]}(\Sigma)$ and denote by $f_0 \in \cZ(A)$ the right support of $e_0 T$. Note that $(1 \ot f_0)p$ is the central support of $\al(e_0)$ inside $p(B(\ell^2(I)) \ovt A)p$. By construction, $\zdim_{A-}(e_0 T \ot_A \overline{T}) \leq e_0$. It follows that the commutant of the left $A$ action on $e_0 T \ot_A \overline{T}$ is a finite von Neumann algebra. A fortiori, $p(B(\ell^2(I)) \ovt A)p(1 \ot f_0)$ is a finite von Neumann algebra. We can thus choose a sequence of projections $q_n \in \cZ(A)f_0$ such that $q_n \recht f_0$ and $p(1 \ot q_n)$ has finite trace for all $n$. Denote by $p_n \in \cZ(A)e_0$ the support of the homomorphism that maps $a \in A e_0$ to $\al(a)(1 \ot q_n)$. It follows that $p_n \recht e_0$.

Since the closure of $\al(Ae_0) (1 \ot q_n)$ inside $L^2(p(B(\ell^2(I)) \ovt A)p)$ has $\zdim_{A-}$ equal to $p_n$, we conclude that $\Sigma p_n \geq p_n$ for all $n$ and thus $\Sigma e_0 \geq e_0$. From the definition of $e_0$, it then follows that $\Sigma e_0 = e_0$ and $e_0 = e_1$ (as defined in the formulation of the lemma), as well as $\Sigma \geq e$ and $f_0 = f_1$. Since $p_n \Sigma = p_n$ for all $n$, it also follows that $\al(Ap_n)(1 \ot q_n)$ is dense in $\al(p_n) L^2(B(\ell^2(I)) \ovt A) p$ for all $n$, because the orthogonal complement has dimension zero. This means that $\al(e_1) = (1 \ot f_1)p$ and that $\al : A e_1 \recht p(B(\ell^2(I)) \ovt A)p(1 \ot f_1)$ is a surjective $*$-isomorphism. So, $e_1 T = T f_1$ and this $A$-bimodule is given by a partial automorphism of $A$.

So the first statement of the lemma is proved. Take $e_2 \in \cZ(A)e$ and $f_2 \in \cZ(A)$ as in the second statement of the lemma. It follows that $e_2 T \ot_A \overline{T} = e_2 T \ot_A \overline{e_2 T}$ and that $\zdim_{A-}(e_2 T \ot_A \overline{T}) = e_2$. So, $e_2 \Sigma = e_2$, meaning that $e_2 \leq e_1$.

Finally take $e_0 \in \cZ(A)$ as in the last statement of the lemma. We have $(\Tr \ot \tau)\al(e_0) = \dim_{-A}(e_0 T) < \infty$. Under the above isomorphism between $T \ot_A \overline{T}$ and $L^2(p(B(\ell^2(I)) \ovt A)p)$, the vector $t_{e_0T}$ corresponds to $\al(e_0)$. So we have to determine the left support $z$ of $\al(e_0) p L^2(B(\ell^2(I)) \ovt A) p \cap \al(Ae_0)^\perp$. A projection $e_3 \in \cZ(A)e_0$ is orthogonal to $z$ if and only if $\al(A e_3)$ is dense in $\al(e_3) p L^2(B(\ell^2(I)) \ovt A) p$. This holds if and only if there exists a projection $f_3 \in \cZ(A)$ such that $\al(e_3) = (1 \ot f_3) p$ and $\al(A e_3) = p(B(\ell^2(I)) \ovt A)p(1 \ot f_3)$. Since this is equivalent with $e_3 \leq e_1$, we have proved that $z = e_0(1-e_1)$.
\end{proof}

\begin{lemma}\label{lem.unique-dec}
Let $(A,\tau)$ be a tracial von Neumann algebra and $(H,J)$ a symmetric $A$-bimodule with left (and thus also, right) support $e \in \cZ(A)$. There is a unique projection $e_1 \in \cZ(A)$ such that $e_1 H = H e_1$, the $A$-bimodule $e_1 H$ is given by a partial automorphism of $A$ and the $A(e-e_1)$-bimodule $(1-e_1)H$ is completely nontrivial.
\end{lemma}
\begin{proof}
By Lemma \ref{lem.char-bimod-partial-aut}, we find projections $e_1,f_1 \in \cZ(A)e$ such that $e_1 H = H f_1$, the $A$-bimodule $e_1 H$ is given by a partial automorphism of $A$ and writing $e_2 := e-e_1$, $f_2 = e-f_1$, the $Ae_2$-$Af_2$-bimodule $e_2 H = H f_2$ is completely nontrivial. Since $H \cong \overline{H}$, we must have $e_1 = f_1$ and $e_2 = f_2$. The uniqueness of $e_1$ can be checked easily.
\end{proof}

%
%

By symmetry, given an $A$-bimodule $H$, we can also define $\Delta_H^r \in \widehat{\cZ(A)}$ characterized by the formula $\tau(\Delta_H^r e) = \dim_{A-}(He)$ for every projection $e \in \cZ(A)$.

\begin{lemma}\label{lem.again-partial-aut}
Let $(A,\tau)$ be a tracial von Neumann algebra and $T$ an $A$-bimodule with left support $e \in \cZ(A)$ and right support $f \in \cZ(A)$. If $\Delta^\ell_T \leq e$ and $\Delta^r_T \leq f$, then $\Delta^\ell_T = e$, $\Delta^r_T = f$ and $T$ is given by a partial automorphism of $A$.
\end{lemma}
\begin{proof}
Let $e_0 \in \cZ(A)e$ be the maximal projection with the following properties: the right support $f_0 \in \cZ(A)f$ of $e_0 T$ satisfies $e_0 T = T f_0$, the $A$-bimodule $e_0T$ is given by a partial automorphism of $A$ and $\Delta^\ell_T = e_0$, $\Delta^r_T = f_0$. We have to prove that $e_0 = e$.

Assume that $e_0$ is strictly smaller than $e$. Since $e_0 T = T f_0$, also $f_0$ is strictly smaller than $f$. Denote $e_1 = e - e_0$ and $f_1 = f-f_0$. Note that $e_1 T = T f_1$. Since $\dim_{-A}(T) = \tau(\Delta^\ell_T) \leq \tau(e) \leq 1$ and similarly $\dim_{A-}(T) \leq 1$, it follows from \cite[Proposition 2.3]{PSV15} that there exists a nonzero $A$-subbimodule $K \subset e_1 T$ with the following properties: $K$ is finitely generated, both as a left Hilbert $A$-module and as a right Hilbert $A$-module, and denoting by $e_2 \in \cZ(A)e_1$ and $f_2 \in \cZ(A)f_1$ the left, resp.\ right, support of $K$, there is a surjective $*$-isomorphism $\al : \cZ(A) f_2 \recht \cZ(A) e_2$ such that $\xi a = \al(a) \xi$ for all $\xi \in K$, $a \in \cZ(A) f_2$.

Denote by $D$ the Radon-Nikodym derivative between $\tau \circ \al$ and $\tau$, so that $\tau(b) = \tau(\al(b)D)$ for all $b \in \cZ(A)f_2$. By a direct computation, we get that
$$\Delta^\ell_K = D \, \al(\zdim_{-A}(K)) \quad\text{and}\quad \al(\Delta^r_K) = D^{-1} \, \zdim_{A-}(K) \; .$$
In particular, we get that
\begin{equation}\label{eq.equal}
\Delta^\ell_K \,  \al(\Delta^r_K) = \zdim_{A-}(K) \, \al(\zdim_{-A}(K)) \; .
\end{equation}
By Lemma \ref{lem.char-bimod-partial-aut} and the computation in the proof of \cite[Lemma 2.2]{PSV15}, we have
\begin{equation}\label{ineq}
\zdim_{A-}(K) \, \al(\zdim_{-A}(K)) = \zdim_{A-}(K \ot_A \overline{K}) \geq e_2 \; .
\end{equation}
Since $\Delta^\ell_K \leq e_2$ and $\Delta^r_K \leq f_2$, in combination with \eqref{eq.equal}, it follows that $\Delta^\ell_K = e_2$ and $\Delta^r_K = f_2$. From \eqref{ineq}, we then also get that $\zdim_{A-}(K \ot_A \overline{K}) = e_2$. By Lemma \ref{lem.char-bimod-partial-aut}, $K$ is given by a partial automorphism of $A$.

Since $e_2 \geq \Delta^\ell_{e_2 T} = \Delta^\ell_K + \Delta^\ell_{e_2 T \ominus K} = e_2 + \Delta^\ell_{e_2 T \ominus K}$, we conclude that $e_2 T \ominus K = \{0\}$. So, $e_2 T = K$ and $e_2 T$ is given by a partial automorphism of $A$. This then contradicts the maximality of $e_0$.
\end{proof}

\begin{lemma}\label{lem.diffuse-element-commutant}
Let $(A,\tau)$ be a tracial von Neumann algebra and $(H,J)$ a symmetric $A$-bimodule. Write $M = \Gamma(H,J,A,\tau)\dpr$. Let $p \in A$ be a projection and $B \subset p A p$ a von Neumann subalgebra such that $B' \cap pAp = \cZ(B)$. Let $K \subset p H$ be a $B$-$A$-subbimodule that is finitely generated as a right Hilbert $A$-module. Assume that $\Delta^\ell_K$ is bounded and satisfies $\Delta^\ell_K \geq p$, as $B$-$A$-bimodule.

Let $(\xi_k)_{k=1}^n$ be a Pimsner-Popa basis for $K$ as a right $A$-module. Then the vectors $\xi_k$ are also left $A$-bounded and using the notation of \eqref{eq.wick}, we define $S \in pMp$ given by
\begin{equation}\label{eq.def-S}
S := \sum_{k=1}^n W(\xi_k,J(\xi_k)) \; .
\end{equation}
Then, $S \in B' \cap pMp$, $S$ is self-adjoint and $S$ is diffuse relative to $B$. More precisely, in the von Neumann algebra $D := \{S\}\dpr$, there exists a unitary $u \in \cU(D)$ satisfying $E_B(u^k)=0$ for all $k \in \Z \setminus \{0\}$.
\end{lemma}


\begin{proof}
Giving a Pimsner-Popa basis $(\xi_k)_{k=1}^n$ for the right Hilbert $A$-module $K$ is the same as defining a right $A$-linear unitary operator $\theta : e(\C^n \ot L^2(A)) \recht K$ for some projection $e \in A^n := M_n(\C) \ot A$, with $\xi_k = \theta(e(e_k \ot 1))$. Define the faithful normal $*$-homomorphism $\al : B \recht e A^n e$ such that $\theta(\al(b) \xi) = b \theta(\xi)$ for all $b \in B$ and $\xi \in e(\C^n \ot L^2(A))$. View $\overline{\C^n} \ot K$ as a $B$-$A^n$-subbimodule of $\overline{\C^n} \ot pH$. Define the vector $\xi \in \overline{\C^n} \ot K$ given by
$$\xi = \sum_{k=1}^n \overline{e_k} \ot \xi_k \; .$$
Then, $b \xi = \xi \al(b)$ for all $b \in B$ and, in particular, $\xi \in (\overline{\C^n} \ot K)e$.

Define the normal positive functional $\om : pAp \recht \C : \om(a) = \langle a \xi , \xi \rangle$. Since $\om$ is $B$-central and $B' \cap pAp = \cZ(B)$, we find $\Delta \in L^1(\cZ(B))^+$ such that $\om(a) = \tau(a \Delta)$ for all $a \in p A p$. But for all projections $q \in B$, we have
$$\tau(q \Delta) = \om(q) = \langle q \xi,\xi \rangle = \langle \xi \al(q) , \xi \rangle = (\Tr \ot \tau)(\al(q)) = \dim_{-A}(qK) \; .$$
This means that $\Delta = \Delta^\ell_K$. Since $\Delta^\ell_K$ is bounded, the vectors $\xi_k \in H$ are left $A$-bounded.

So, the vectors $\xi_k$ are both left and right $A$-bounded, so that the operator $S$ given by \eqref{eq.def-S} is a well defined element of $pMp$. Since
$$S = \sum_{k=1}^n \bigl(\ell(\xi_k) \ell(J(\xi_k)) + \ell(\xi_k)\ell(\xi_k)^* + \ell(J(\xi_k))^* \ell(\xi_k)^*\bigr) \; ,$$
we get that $S = S^*$. From this formula, we also get that $S$ commutes with $B$. Put $S_1 := \Delta + S$. Since $\Delta \in \cZ(B)$, it suffices to prove that $S_1$ is diffuse relative to $B$.

Write $A_1 = pAp$ and $A_2 = e A^n e$. Equip $A_1$ and $A_2$ with the non normalized traces given by restricting $\tau$ to $A_1$ and $\Tr \ot \tau$ to $A_2$. View $\xi$ as a vector in the $A_1$-$A_2$-bimodule $(\overline{\C^n} \ot p H)e$ and note that
$$\langle \xi,\xi \rangle_{A_2} = e \;\; , \;\; {}_{A_1} \langle \xi, \xi \rangle = \Delta \; .$$
Denote $L := (\overline{\C^n} \ot pH) e$. Recall that we view $L$ as an $A_1$-$A_2$-bimodule and that $\xi \in L$. Write $L' := e(\C^n \ot H p)$, view $L'$ as an $A_2$-$A_1$-bimodule and note that the anti-unitary operator
$$J_1 : L \recht L' : J_1\bigl(\sum_{k=1}^n \overline{e_k} \ot \mu_k\bigr) = \sum_{k=1}^n e_k \ot J(\mu_k)$$
satisfies $J_1(a \mu b) = b^* J_1(\mu) a^*$ for all $\mu \in L$, $a \in A_1$ and $b \in A_2$. Define $\xi' \in L'$ given by $\xi' = J_1(\xi) \Delta^{-1/2}$. Then $\xi'$ satisfies the following properties.
$$\langle \xi',\xi'\rangle_{A_1} = p \;\; , \;\; {}_{A_2} \langle \xi',\xi'\rangle = \al(\Delta^{-1}) \;\;\text{and}\;\; \al(b) \xi' = \xi' b \;\; \forall b \in B \; .$$

Define the Hilbert spaces
\begin{align*}
L\even &= L^2(A_1) \oplus \bigoplus_{m=1}^\infty \bigl( L \ot_{A_2} L' \bigr)^{\ot^m_{A_1}} \; ,   \\
L\odd &= L' \ot_{A_1} L\even = \bigoplus_{m=0}^\infty \bigl( L' \ot_{A_1}  \bigl( L \ot_{A_2} L' \bigr)^{\ot^m_{A_1}}\bigr) \; .
\end{align*}
Note that $L\even$ is an $A_1$-bimodule, while $L\odd$ is an $A_2$-$A_1$-bimodule.
Then,
\begin{equation}\label{eq.W}
W := \ell(\xi') \Delta^{1/2} + \ell(\xi)^*
\end{equation}
is a well defined bounded operator from $L\even$ to $L\odd$ and $W^* W \in B(L\even)$.

Using the natural isometry $L \ot_{A_2} L' \hookrightarrow p(H \ot_A H)p$, we define the isometry $V : L\even \recht p L^2(M) p$ given as the direct sum of the compositions of
$$\bigl( L \ot_{A_2} L' \bigr)^{\ot^m_{A_1}} \hookrightarrow \bigl( p (H \ot_A H) p\bigr)^{\ot^m_{A_1}} \hookrightarrow p \bigl( H^{\ot^{2m}_A}\bigr) p \; .$$
Then $V$ is $A_1$-bimodular and
\begin{equation}\label{eq.first-identif}
V \, W^* W = S_1 \, V \; .
\end{equation}
To compute the $*$-distribution of $B \cup \{S_1\}$ w.r.t.\ the trace $\tau$, it is thus sufficient to compute the $*$-distribution of $B \cup \{W^* W\}$ acting on $L\even$ and w.r.t.\ the vector functional implemented by $p \in L^2(A_1) \subset L\even$.

Define the closed subspaces $L^0\even \subset L\even$ and $L^0\odd \subset L\odd$ given as the closed linear span
\begin{align*}
L^0\even & = \cspan \{L^2(B) , (\xi \ot_{A_2} \xi')^{\ot_{A_1}^m} B \mid m \geq 1 \} \; , \\
L^0\odd & = \cspan \{(\xi' \ot_{A_1} (\xi \ot_{A_2} \xi')^{\ot_{A_1}^m}) B \mid m \geq 0 \} \; .
\end{align*}
Since $\xi \ot_{A_2} \xi'$ is a $B$-central vector and since $\langle \xi, \xi \rangle_{A_2} = e$ and $\langle \xi' ,\xi' \rangle_{A_1} = p$, we find that $W(L^0\even) \subset L^0\odd$ and $W^*(L^0\odd) \subset L^0\even$. So to compute the $*$-distribution of $B \cup \{W^* W\}$, we may restrict $B$ and $W^* W$ to $L^0\even$.

Consider the full Fock space $\cF(\C^2)$ of the $2$-dimensional Hilbert space $\C^2$, with creation operators $\ell_1 = \ell(e_1)$ and $\ell_2 = \ell(e_2)$ given by the standard basis vectors $e_1,e_2 \in \C^2$. Denote by $\eta$ the vector state on $B(\cF(\C^2))$ implemented by the vacuum vector $\Omega \in \cF(\C^2)$. For every $\lambda \geq 1$, consider the operator $X(\lambda) \in B(\cF(\C^2))$ given by $X(\lambda) = \sqrt{\lambda} \ell_2 + \ell_1^*$. We find that $X(\lambda)^* X(\lambda) = \lambda y^* y$ with $y = \ell_2 + \lambda^{-1/2} \ell_1^*$. It then follows from \cite[Lemma 4.3 and discussion after Definition 4.1]{Sh96} that the spectral measure of $X(\lambda)^* X(\lambda)$ has no atoms. Also for every $\lambda \geq 1$, $\eta$ is a faithful state on $\{X(\lambda)^* X(\lambda)\}\dpr$.

Identify $\cZ(B) = L^\infty(Z,\mu)$ for some standard probability space $(Z,\mu)$. View $\Delta$ as a bounded function from $Z$ to $[1,+\infty)$ and define $Y \in B(\cF(\C^2)) \ovt L^\infty(Z,\mu)$ given by $Y(z) = X(\Delta(z))$. We can view $Y$ as an element of $B(\cF(\C^2)) \ovt B$ acting on the Hilbert space $\cF(\C^2) \ot L^2(B)$. Also, $\eta \ot \tau$ is faithful on $(1 \ot B \cup \{Y^* Y\})\dpr$. Define the isometry
$$U : L^0\even \recht \cF(\C^2) \ot L^2(B) : U\bigl((\xi \ot_{A_2} \xi')^{\ot_{A_1}^m} b\bigr) = (e_1 \ot e_2)^{\ot m} \ot b \; .$$
By construction, $U W^* W = Y^*Y U$ and $U$ is $B$-bimodular. It follows that the $*$-distribution of $B \cup \{S_1\}$ w.r.t.\ $\tau$ equals the $*$-distribution of $1 \ot B \cup \{Y^*Y\}$ w.r.t.\ $\eta \ot \tau$. So there is a unique normal $*$-isomorphism
$$\Psi : (1 \ot B \cup \{Y^*Y\})\dpr \recht (B \cup \{S_1\})\dpr$$
satisfying $\Psi(1 \ot b) = b$ for all $b \in B$ and $\Psi(Y^* Y) = S_1$. Also, $\tau \circ \Psi = \eta \ot \tau$. Since for all $z \in Z$, the spectral measure of $Y(z)^* Y(z)$ has no atoms, there exists a unitary $v \in \{Y^*Y\}\dpr$ such that $(\eta \ot \tau)((1 \ot b) v^k)=0$ for all $b \in B$ and $k \in \Z \setminus \{0\}$. Taking $u = \Psi(v)$, the lemma is proved.
\end{proof}

\begin{definition}[{\cite[Definition A.2]{Va07}}]\label{def.ess-finite-index}
A von Neumann subalgebra $P$ of a tracial von Neumann algebra $(Q,\tau)$ is said to be of \emph{essentially finite index} if there exist projections $q \in P' \cap Q$ arbitrarily close to $1$ such that $P q \subset qQq$ has finite Jones index.
\end{definition}

To make the connection with \cite[Lemma 1.6]{Io11}, note that $P \subset Q$ is essentially of finite index if and only if $qQq \prec_{qQq} Pq$ for every nonzero projection $q \in P' \cap Q$.

\begin{lemma}\label{lem.complement}
Let $(A,\tau)$ be a tracial von Neumann algebra and $(H,J)$ a symmetric $A$-bimodule. Write $M = \Gamma(H,J,A,\tau)\dpr$.

Let $p \in A$ be a projection and $B \subset p A p$ a von Neumann subalgebra such that $B' \cap pAp = \cZ(B)$ and such that $\cN_{pAp}(B)\dpr$ has essentially finite index in $pAp$. Let $K_1 \subset pH$ be a $B$-$A$-subbimodule satisfying the following three properties.
\begin{enumlist}
\item $K_1$ is a direct sum of $B$-$A$-subbimodules of finite right $A$-dimension.
\item The left action of $B$ on $K_1$ is faithful.
\item The $A$-bimodule $\overline{AK_1}$ is left weakly mixing.
\end{enumlist}
Then there exists a diffuse abelian von Neumann subalgebra $D \subset B' \cap pMp$ that is in tensor product position w.r.t.\ $B$. More precisely, there exists a unitary $u \in B' \cap pMp$ such that $E_B(u^k) = 0$ for all $k \in \Z \setminus \{0\}$.
\end{lemma}

\begin{proof}
%
We claim that for every $\eps > 0$, there exists a projection $z \in \cZ(B)$ with $\tau(p-z) < \eps$ and a $B$-$A$-subbimodule $L \subset z H$ such that $L$ is finitely generated as a right Hilbert $A$-module and such that $\Delta^\ell_L$ is bounded and satisfies $\Delta^\ell_L \geq z$. To prove this claim, denote $K := \overline{A K_1}$ and let $(K_i)_{i \in I}$ be a maximal family of mutually orthogonal nonzero $B$-$A$-subbimodules of $pK$ that are finitely generated as a right $A$-module. Denote by $R$ the closed linear span of all $K_i$. Whenever $u \in \cN_{pAp}(B)$ and $i \in I$, also $u K_i$ is a $B$-$A$-subbimodule of $pK$ that is finitely generated as a right $A$-module. By the maximality of the family $(K_i)_{i \in I}$, we get that $u K_i \subset R$. So, $u R = R$ for all $u \in \cN_{pAp}(B)$. Writing $P := \cN_{pAp}(B)\dpr$, we conclude that $R$ is a $P$-$A$-subbimodule of $pK$.

Since $P \subset pAp$ is essentially of finite index and since $\bim{A}{K}{A}$ is left weakly mixing, Lemma \ref{lem.infinite-dim} says that for every projection $q \in P$, the right $A$-module $qR$ is either $\{0\}$ or of infinite right $A$-dimension. By the assumptions of the lemma and the maximality of the family $(K_i)_{i \in I}$, the left $B$-action on $R$ is faithful. So $qL \neq \{0\}$ and thus $\dim_{-A}(qL) = \infty$ for every nonzero projection $q \in B$. This means that for every nonzero projection $q \in B$,
$$\sum_{i \in I} \tau(q \Delta^\ell_{K_i}) = \sum_{i \in I} \dim_{-A}(q K_i) = \dim_{-A}(q R) = \infty \; .$$
So we can find a projection $z \in \cZ(B)$ and a finite subset $I_0 \subset I$ such that $\tau(p-z) < \eps$ and such that the operator $\Delta := \sum_{i \in I_0} \Delta^\ell_{K_i} z$ is bounded and satisfies $\Delta \geq z$. Defining $L = \sum_{i \in I_0} z K_i$, the claim is proved.

Combining the claim with Lemma \ref{lem.diffuse-element-commutant}, we find for every $\eps > 0$, a projection $z \in \cZ(B)$ with $\tau(p-z) < \eps$ and a unitary $u \in (Bz)' \cap zMz$ such that $E_B(u^k) = 0$ for all $k \in \Z \setminus \{0\}$. So, we find projections $z_n \in \cZ(B)$ and unitaries $u_n \in (Bz_n)' \cap z_n M z_n$ such that $E_B(u_n^k) = 0$ for all $k \in \Z \setminus \{0\}$ and such that $\bigvee_n z_n = p$. We can then choose projections $z'_n \in \cZ(B)$ with $z'_n \leq z_n$ and $\sum_n z'_n = p$. Defining $u = \sum_n u_n z'_n$, we have found a unitary in $B' \cap pMp$ satisfying $E_B(u^k) = 0$ for all $k \in \Z \setminus \{0\}$. So, the lemma is proved.
\end{proof}

Above we also needed the following two lemmas.

\begin{lemma}\label{lem.no-complement}
Let $(N,\tau)$ be a tracial von Neumann algebra and $B \subset N$ an abelian von Neumann subalgebra. Assume that $D \subset B' \cap N$ is a diffuse abelian von Neumann subalgebra that is in tensor product position w.r.t.\ $B$. Then there is no nonzero projection $q \in B' \cap N$ satisfying $q (B' \cap N) q = B q$.
\end{lemma}
\begin{proof}
Put $P = B' \cap N$ and assume that $q \in P$ is a nonzero projection such that $qPq = B q$. Note that $B \subset \cZ(P)$ because $B$ is abelian. Take a nonzero projection $z \in \cZ(P)$ such that $z = \sum_{i=1}^n v_i v_i^*$ where $v_1,\ldots,v_n$ are partial isometries in $P q$. Note that $z q \neq 0$ and write $p = zq$. Then,
$$P p = P z q = z P q = \lspan \{v_i q P q \mid i=1,\ldots ,n\} = \lspan \{ v_i B \mid i=1,\ldots,n\} \; .$$
So, $L^2(P) p$ is finitely generated as a right Hilbert $B$-module. Define $Q = B \vee D$ and denote by $e \in Q$ the support projection of $E_Q(p)$. Then $\xi \mapsto \xi p$ is an injective right $B$-linear map from $L^2(Q)e$ to $L^2(P) p$. So also $L^2(Q)e$ is finitely generated as a right Hilbert $B$-module. Since $Q \cong B \ovt D$ with $D$ diffuse and since $e$ is a nonzero projection in $Q \cong B \ovt D$, this is absurd.
\end{proof}

\begin{lemma}\label{lem.infinite-dim}
Let $(A,\tau)$ be a tracial von Neumann algebra and $\bim{A}{K}{A}$ an $A$-bimodule that is left weakly mixing. Let $p \in A$ be a projection and $P \subset pAp$ a von Neumann subalgebra that is essentially of finite index (see Definition \ref{def.ess-finite-index}). If $L \subset pK$ is a $P$-$A$-subbimodule and $q \in P$ is a projection such that $q L \neq \{0\}$, then the right $A$-dimension of $qL$ is infinite.
\end{lemma}
\begin{proof}
Assume for contradiction that $q \in P$ is a projection such that $q L$ is nonzero and such that $qL$ has finite right $A$-dimension.

Since $P \subset pAp$ is essentially of finite index, there exist projections $p_1 \in P' \cap pAp$ that lie arbitrarily close to $p$ such that $Ap_1$ is finitely generated as a right $Pp_1$ module (purely algebraically using a Pimsner-Popa basis, see e.g.\ \cite[A.2]{Va07}). There also exist central projections $z \in \cZ(P)$ that lie arbitrarily close to $p$ such that $P z q$ is finitely generated as a right $qPq$-module. Take such $p_1$ and $z$ with $p_1 z q L \neq \{0\}$. Then $A p_1 z q$ is finitely generated as a right $qPq$-module. Therefore, the closed linear span of $A p_1 z q L$ is a nonzero $A$-subbimodule of $K$ having finite right $A$-dimension. This contradicts the left weak mixing of $\bim{A}{H}{A}$.
\end{proof}

\section{\boldmath Compact groups, free subsets, $c_0$ probability measures and the proof of Theorem \ref{thm.main-B}}

For every second countable compact group $K$ with Haar probability measure $\mu$ and for every symmetric probability measure $\nu$ on $K$, we consider $A=L^\infty(K,\mu)$, the $A$-bimodule $H_\nu = L^2(K\times K,\mu \times \nu)$ given by \eqref{eq.my-bim-compact-group} and the symmetry $J_\nu : H_\nu \recht H_\nu$ given by \eqref{eq.my-symm}. We put $M = \Gamma(H_\nu,J_\nu,A,\mu)\dpr$.

In Proposition \ref{prop.char} below, we characterize when the bimodule $H_\nu$ is mixing (so that $M$ becomes strongly solid by Corollary \ref{cor.rel-strong-solid}) and when $A \subset M$ is an $s$-MASA. For the latter, the crucial property will be that the support $S$ of $\nu$ is of the form $S = F \cup F^{-1}$ where $F \subset K$ is a closed subset that is \emph{free} in the following sense.

\begin{definition}\label{def.free}
A subset $F$ of a group $G$ is called \emph{free} if
$$g_1^{\eps_1} \cdots g_n^{\eps_n} \neq e$$
for all nontrivial \emph{reduced words}, i.e.\ for all $n \geq 1$ and all $g_1,\ldots,g_n \in F$, $\eps_1,\ldots,\eps_n \in \{\pm 1\}$ satisfying $\eps_i = \eps_{i+1}$ whenever $1 \leq i \leq n-1$ and $g_i = g_{i+1}$.
\end{definition}

On the other hand, the mixing property of $H_\nu$ will follow from the following $c_0$ condition on the measure $\nu$.

Whenever $K$ is a compact group, we denote by $\lambda : K \recht \cU(L^2(K))$ the left regular representation. For every probability measure $\nu$ on $K$ and every unitary representation $\pi : K \recht \cU(H)$, we denote
$$\pi(\nu) = \int_K \pi(x) \; d\nu(x) \; .$$

\begin{definition}\label{def.c0}
A probability measure $\nu$ on a compact group $K$ is said to be $c_0$ if the operator $\lambda(\nu) \in B(L^2(K))$ is compact.
\end{definition}

Note that $\nu$ is $c_0$ if and only if $\lambda(\nu)$ belongs to the reduced group C$^*$-algebra $C^*_r(K)$. Also, since the regular representation of $K$ decomposes as the direct sum of all irreducible representations of $K$, each appearing with multiplicity equal to its dimension, we get that a probability measure $\nu$ is $c_0$ if and only if
$$\lim_{\pi \in \Irr(K), \pi \recht \infty} \|\pi(\nu)\| = 0 \; ,$$
i.e.\ if and only if the map $\Irr(K) \recht \R : \pi \mapsto \|\pi(\nu)\|$ is $c_0$. In particular, when $K$ is an abelian compact group, a probability measure $\nu$ on $K$ is $c_0$ if and only if the Fourier transform of $\nu$ is a $c_0$ function on $\widehat{K}$.

\begin{proposition}\label{prop.char}
Let $K$ be a second countable compact group $K$ with Haar probability measure $\mu$. Put $A = L^\infty(K,\mu)$. Let $\nu$ be a symmetric probability measure on $K$ without atoms. Define the $A$-bimodule $H_\nu$ with symmetry $J_\nu$ by \eqref{eq.my-bim-compact-group} and \eqref{eq.my-symm}. Denote by $M = \Gamma(H_\nu,J_\nu,A,\mu)\dpr$ the associated tracial von Neumann algebra. Let $S$ be the support of $\nu$, i.e.\ the smallest closed subset of $K$ with $\nu(S) = 1$.
\begin{enumlist}
\item The bimodule $H_\nu$ is weakly mixing, $A \subset M$ is a singular MASA, $M$ has no Cartan subalgebra and $A \subset M$ is a maximal amenable subalgebra.
\item The von Neumann algebra $M$ has no amenable direct summand. The center $\cZ(M)$ of $M$ equals $L^\infty(K/K_0)$ where $K_0 \subset K$ is the closure of the subgroup generated by $S$. So if $S$ topologically generates $K$, then $M$ is a nonamenable II$_1$ factor.
\item If $S$ is of the form $S = F \cup F^{-1}$ where $F \subset K$ is a closed subset that is free in the sense of Definition \ref{def.free}, then $A \subset M$ is an $s$-MASA.
\item If $\nu$ is $c_0$ in the sense of Definition \ref{def.c0}, then the bimodule $H_\nu$ is mixing. So then, $M$ is strongly solid and whenever $B \subset M$ is an amenable von Neumann subalgebra for which $B \cap A$ is diffuse, we have $B \subset A$.
\end{enumlist}
\end{proposition}
\begin{proof}
1.\ Note that
\begin{equation}\label{eq.Hnu-n}
H_\nu^{\ot^n_A} \cong L^2(K \times \underbrace{K \times \cdots \times K}_{\text{$n$ times}},\mu \times \underbrace{\nu \times \cdots \times \nu}_{\text{$n$ times}})
\end{equation}
with the $A$-bimodule structure given by
$$(F \cdot \xi \cdot G)(x,y_1,\ldots,y_n) = F(xy_1\cdots y_n) \, \xi(x,y_1,\ldots,y_n) \, G(x) \; .$$
Define $D \subset K \times K$ given by $D = \{(y,y^{-1}) \mid y \in K\}$. Since $\nu$ has no atoms, we have $(\nu \times \nu)(D) = 0$. It then follows that $H_\nu \ot_A H_\nu$ has no nonzero $A$-central vectors. By Proposition \ref{prop5}, the $A$-bimodule $H_\nu$ is weakly mixing. So also $L^2(M) \ominus L^2(A)$ is a weakly mixing $A$-bimodule, implying that $\cN_M(A) \subset A$. So, $A \subset M$ is a MASA and this MASA is singular. By Theorem \ref{thm.absence-Cartan}, $M$ has no Cartan subalgebra. By Theorem \ref{thm.max-amen}, we get that $A \subset M$ is a maximal amenable subalgebra.

2.\ Since $H_\nu$ is weakly mixing, we get from Theorem \ref{thm.max-amen} that $M$ has no amenable direct summand and that $\cZ(M)$ consists of all $a \in A$ satisfying $a \cdot \xi = \xi \cdot a$ for all $\xi \in H_\nu$. It is then clear that $L^\infty(K/K_0) \subset \cZ(M)$. To prove the converse, fix $a \in A$ with $a \cdot \xi = \xi \cdot a$ for all $\xi \in H_\nu$. We find in particular that $a(xy) = a(x)$ for $\mu \times \nu$-a.e.\ $(x,y) \in K \times K$. Let $\cU_n$ be a decreasing sequence of basic neighborhoods of $e$ in $K$. Define the functions $b_n$ given by
$$b_n(y) = \mu(\cU_n)^{-1} \int_{\cU_n} a(xy) \, d\mu(x) \; .$$
For every fixed $n$, the functions $b_n$ still satisfy $b_n(xy) = b(x)$ for $\mu \times \nu$-a.e.\ $(x,y) \in K \times K$. But the functions $b_n$ are continuous. It follows that $b_n(xy) = b_n(x)$ for all $x \in K$ and all $y \in S$. So, $b_n \in C(K/K_0)$. Since $\lim_n \|b_n -a\|_1 = 0$, we get that $a \in L^\infty(K/K_0)$.

3.\ Denote by $W_n \subset (F \cup F^{-1})^n$ the subset of reduced words of length $n$. Since $\nu$ has no atoms, we find that $\nu^n(W_n) = 1$. Denote by $\pi_n : K^n \recht K$ the multiplication map and put $S_n := \pi_n(W_n)$. Since $F$ is free, the subsets $S_n \subset K$ are disjoint. By freeness of $F$, we also have that the restriction of $\pi_n$ to $W_n$ is injective. Define the probability measures $\nu_n := (\pi_n)_*(\nu^n)$ and then $\eta = \frac{1}{2} \delta_0 + \sum_{n=1}^\infty 2^{-n-1} \nu_n$. Using \eqref{eq.Hnu-n}, it follows that $\bim{A}{L^2(M)}{A}$ is isomorphic with the $A$-bimodule
$$L^2(K \times K, \mu \times \eta) \quad\text{with}\quad (F \cdot \xi \cdot G)(x,y) = F(xy) \, \xi(x,y) \, G(x) \; .$$
So, $\bim{A}{L^2(M)}{A}$ is a cyclic bimodule and $A \subset M$ is an $s$-MASA.

4.\ Define $\xi_0 \in H_\nu$ by $\xi_0(x,y) = 1$ for all $x,y \in K$. Denote by $\vphi : A \recht A$ the completely positive map given by $\vphi(a) = \langle \xi_0, a \xi_0\rangle_A$. To prove that $H_\nu$ is mixing, it is sufficient to prove that $\lim_n \|\vphi(a_n)\|_2 = 0$ whenever $(a_n)$ is a bounded sequence in $A$ that converges weakly to $0$. Denoting by $\rho : K \recht L^2(K)$ the right regular representation, we get that $\vphi(a) = \rho(\nu)(a)$ for all $a \in A \subset L^2(K)$. Since $\rho(\nu)$ is a compact operator, we indeed get that $\lim_n \|\rho(\nu)(a_n)\|_2 = 0$. So, $H_\nu$ is a mixing $A$-bimodule. By Corollary \ref{cor.rel-strong-solid}, $M$ is strongly solid. The remaining statement follows from Theorem \ref{thm.max-amen}.
\end{proof}

\begin{remark}\label{rem.free-bogol-abelian}
In the special case where $K$ is abelian, we identify $L^\infty(K,\mu) = L(G)$, with $G:=\widehat{K}$ being a countable abelian group. Then the symmetric $L^\infty(K,\mu)$-bimodule $H_\nu$ given by \eqref{eq.my-bim-compact-group} and \eqref{eq.my-symm} is isomorphic with the symmetric $L(G)$-bimodule associated, as in Remark \ref{rem.bogol-crossed}, with the cyclic orthogonal representation of $G$ with spectral measure $\nu$. In particular, as in Remark \ref{rem.bogol-crossed}, the von Neumann algebras $M=\Gamma(H_\nu,J_\nu,L^\infty(K),\mu)\dpr$ can also be realized as a free Bogoljubov crossed product by the countable abelian group $G$. In this way, Proposition \ref{prop.char} generalizes the results of \cite{HS09,Ho12a}. Note however that for a free Bogoljubov crossed product $M = \Gamma(K_\R)\dpr \rtimes G$ with $G$ abelian, the subalgebra $L(G) \subset M$ is \emph{never} an $s$-MASA. So our more general construction is essential to prove Theorem \ref{thm.main-B}.

For non abelian compact groups $K$, we can still view $K = \widehat{G}$, but $G$ is no longer a countable group, rather a discrete Kac algebra. It is then still possible to identify the II$_1$ factors $M$ in Proposition \ref{prop.char} with a crossed product $\Gamma(K_\R)\dpr \rtimes G$, where the discrete Kac algebra action of $G$ on $\Gamma(K_\R)\dpr$ is the free Bogoljubov action associated in \cite{Va02} with an orthogonal corepresentation of the quantum group $G$.
\end{remark}

The main result of this section says that in certain sufficiently non abelian compact groups $K$, one can find ``large'' free subsets $F \subset K$, where ``large'' means that $F$ carries a non atomic probability measure that is $c_0$. We conjecture that the compact Lie groups $\SO(n)$, $n \geq 3$, admit free subsets carrying a $c_0$ probability measure. For our purposes, it is however sufficient to prove that these exist in more ad hoc groups.

For every prime number $p$, denote by $\Gamma_p$ the finite group $\Gamma_p = \PGL_2(\Z/p\Z)$. The following is the main result of this section. Recall that the support of a probability measure $\nu$ on a compact space $K$ is defined as the smallest closed subset $S \subset K$ with $\nu(S) = 1$.

\begin{theorem}\label{thm.free-c0}
There exists a sequence of prime numbers $p_n$ tending to infinity, a closed free subset $F \subset K := \prod_{n=1}^\infty \Gamma_{p_n}$ topologically generating $K$ and a symmetric, non atomic, $c_0$ probability measure $\nu$ on $K$ whose support equals $F \cup F^{-1}$.
\end{theorem}

We then immediately get:

\begin{proof}[Proof of Theorem \ref{thm.main-B}]
Take $K$ and $\nu$ as in Theorem \ref{thm.free-c0}. Denote by $M$ the associated von Neumann algebra with abelian subalgebra $A \subset M$ as in Proposition \ref{prop.char}. By Proposition \ref{prop.char}, we get that $M$ is a nonamenable, strongly solid II$_1$ factor and that $A \subset M$ is an $s$-MASA.
\end{proof}

Before proving Theorem \ref{thm.free-c0}, we need some preparation.

The Alon-Roichman theorem \cite{AR92} asserts that the Cayley graph given by a random and independent choice of $k \geq c(\eps) \, \log |G|$ elements in a finite group $G$ has expected second eigenvalue at most $\eps$, with the normalization chosen so that the largest eigenvalue is $1$. In \cite[Theorem~2]{LR04}, a simple proof of that result was given. The same proofs yields the following result. For completeness, we provide the argument.

Whenever $G$ is a group, $\pi : G \recht \cU(H)$ is a unitary representation and $g_1,\ldots,g_k \in G$, we write
\begin{equation}\label{eq.mean}
\pi(g_1,\ldots,g_k) := \frac{1}{k} \sum_{j=1}^k \pi(g_j) \; .
\end{equation}

\begin{lemma}[{\cite{LR04}}]\label{lem.gap}
Let $G_n$ be a sequence of finite groups and $k_n$ a sequence of positive integers such that $k_n / \log |G_n| \recht \infty$. For every $\eps > 0$ and for a uniform and independent choice of $k_n$ elements $g_1,\ldots,g_{k_n} \in G_n$, we have that
$$
\lim_{n \recht \infty} \rP\Bigl( \; \|\pi(g_1,\ldots,g_{k_n})\| \leq \eps \;\; \text{for all}\;\; \pi \in \Irr(G_n) \setminus \{\counit\} \; \Bigr) = 1 \; .
$$
\end{lemma}
\begin{proof}
Fix a finite group $G$ and a positive integer $k$. Let $g_1,\ldots,g_k$ be a uniform and independent choice of elements of $G$. Denote by $\lambda_0 : G \recht \cU(\ell^2(G) \ominus \C 1)$ the regular representation restricted to $\ell^2(G) \ominus \C 1$. Put $d = |G| - 1$. Both
$$T(g_1,\ldots,g_k) = \frac{1}{k} \sum_{j=1}^k \frac{\lambda_0(g_j) + \lambda_0(g_j)^*}{2} \quad\text{and}\quad S(g_1,\ldots,g_k) = \frac{1}{k} \sum_{j=1}^k \frac{i \lambda_0(g_j) - i\lambda_0(g_j)^*}{2}$$
are sums of $k$ independent self-adjoint $d \times d$ matrices of norm at most $1$ and having expectation $0$. We apply \cite[Theorem 19]{AW01} to the independent random variables
$$X_j = \frac{2+\lambda_0(g_j)+\lambda_0(g_j)^*}{4} \; ,$$
satisfying $0 \leq X_j \leq 1$ and having expectation $1/2$. We conclude that for every $0 \leq \eps \leq 1/2$,
$$\rP\bigl( \, \|T(g_1,\ldots,g_k)\| \leq \eps \, \bigr) = \rP\Bigl( \, (1-\eps) \frac{1}{2} \leq \frac{1}{k} \sum_{j=1}^k X_j \leq (1+\eps) \frac{1}{2} \, \Bigr) \geq 1 - 2 d \, \exp\bigl(- k \frac{\eps^2}{4 \log 2}\bigr) \; .$$
The same estimate holds for $S(g_1,\ldots,g_k)$. Since $\lambda_0(g_1,\ldots,g_k) = T(g_1,\ldots,g_k) - i S(g_1,\ldots,g_k)$ and since $\lambda_0$ is the direct sum of all nontrivial irreducible representations of $G$ (all appearing with multiplicity equal to their dimension), we conclude that
$$\rP\bigl( \, \|\pi(g_1,\ldots,g_k) \| \leq \eps \;\; \text{for all}\;\; \pi \in \Irr(G) \setminus \{\counit\} \, \bigr) \geq 1 - 4 |G| \, \exp\bigl(- k \frac{\eps^2}{16 \log 2}\bigr) \; .$$
Taking $G = G_n$, $k=k_n$ and $n \recht \infty$, our assumption that $k_n / \log |G_n| \recht \infty$ implies that for every fixed $\eps > 0$,
$$|G_n| \, \exp\bigl(- k_n \frac{\eps^2}{16 \log 2}\bigr) \recht 0$$
and thus the lemma follows.
\end{proof}

On the other hand in \cite{GHSSV07}, it is proven that random Cayley graphs of the groups $\PGL_2(\Z/p\Z)$ have large girth. More precisely, we say that elements $g_1,\ldots,g_k$ in a group $G$ satisfy no relation of length $\leq \ell$ if every nontrivial reduced word of length at most $\ell$ with letters from $g_1^{\pm 1},\ldots,g_k^{\pm 1}$ defines a nontrivial element in $G$.

The estimates in the proof of \cite[Lemma~10]{GHSSV07} give the following result. Again for completeness, we provide the argument.

\begin{lemma}[{\cite{GHSSV07}}]\label{lem.girth}
Let $p_n$ be a sequence of prime numbers tending to infinity and let $k_n$ be a sequence of positive integers such that $\log k_n / \log p_n \recht 0$. Put $\Gamma_{p_n} = \PGL_2(\Z/p_n \Z)$. For every $\ell > 0$ and for a uniform and independent choice of $k_n$ elements $g_1,\ldots,g_{k_n} \in \Gamma_{p_n}$, we have that
$$\lim_{n \recht \infty} \rP\Bigl( \; g_1,\ldots,g_{k_n} \;\;\text{satisfy no relation of length $\leq \ell$} \; \Bigr) = 1 \; .$$
\end{lemma}
\begin{proof}
Let $G$ be a group. A \emph{law} of length $\ell$ in $G$ is a nontrivial element $w$ in a free group $\F_n$ such that $w$ has length $\ell$ and $w(g_1,\ldots,g_n) = e$ for all $g_1,\ldots,g_n \in G$. For example, if $G$ is abelian, the element $w = a b a^{-1} b^{-1}$ of $\F_2$ defines a law of length $4$ in $G$. Since the labeling of the generators does not matter, any law of length $\ell$ can be defined by a nontrivial element of $\F_n$ with $n \leq \ell$. In particular, there are only finitely many possible laws of a certain length $\ell$.

Since $\F_\infty \hookrightarrow \F_2 \hookrightarrow \PSL_2(\Z)$, the group $\PSL_2(\Z)$ satisfies no law. For every prime number $p$, write $\Gamma_p = \PGL_2(\Z/p\Z)$. Using the quotient maps $\PSL_2(\Z) \recht \PSL_2(\Z/p\Z)$, we get that a given nontrivial element $w \in \F_n$ can be a law for at most finitely many $\Gamma_p$. So, for every $\ell > 0$, we get that $\Gamma_p$ satisfies no law of length $\leq \ell$ for all large enough primes $p$. (Note that \cite[Proposition 11]{GHSSV07} provides a much more precise result.)

Let $w = g_{i_1}^{\eps_1} \cdots g_{i_\ell}^{\eps_\ell}$ with $i_j \in \{1,\ldots,k\}$ and $\eps_j \in \{\pm 1\}$ be a reduced word of length $\ell$ in $g_1^{\pm 1},\ldots,g_k^{\pm 1}$. Let $p$ be a prime number and assume that $w$ is not a law of $\Gamma_p$. With the same argument as in the proof of \cite[Lemma~10]{GHSSV07}, we now prove that for a uniform and independent choice of $g_1,\ldots,g_k \in \Gamma_p$, we have that
\begin{equation}\label{eq.est-l}
\rP\bigl( \, w(g_1,\ldots,g_k) = e \;\text{in}\; \Gamma_p \, \bigr) \leq \frac{\ell}{p} \, \bigl(1+\frac{1}{p-1}\bigr)^{3k} \; .
\end{equation}
Denote $F_p = \Z/p\Z$, not to be confused with the free group $\F_p$. Write $G_p = GL_2(F_p) \subset F_p^{2 \times 2}$. Define the map
$$W : \bigl(F_p^{2\times 2}\bigr)^k \recht F_p^{2\times 2} : W(a_1,\ldots,a_k) = b_{i_1} \cdots b_{i_\ell}$$
where $b_{i_j} = a_{i_j}$ when $\eps_j = 1$ and $b_{i_j}$ equals the adjunct matrix of $a_{i_j}$ when $\eps_j = -1$. Note that the four components $W_{st}$, $s,t \in \{1,2\}$, of the map $W$ are polynomials of degree at most $\ell$ in the $4k$ variables $a \in \bigl(F_p^{2\times 2}\bigr)^k$. Define the subset $\cW \subset \bigl(F_p^{2\times 2}\bigr)^k$ given by
\begin{align*}
\cW &= \bigl\{ a \in \bigl(F_p^{2\times 2}\bigr)^k \bigm| W(a) \;\text{is a multiple of the identity matrix}\,\bigr\} \\
 &= \bigl\{ a \in \bigl(F_p^{2\times 2}\bigr)^k \bigm| W_{11}(a) - W_{22}(a) = W_{12}(a) = W_{21}(a) = 0 \bigr\} \; .
\end{align*}
We also define $\cV = \cW \cap (G_p)^k$ and
$$\cU = \{ g \in (\Gamma_p)^k \mid w(g_1,\ldots,g_k) = e \;\text{in}\; \Gamma_p \} \; .$$
The quotient map $G_p \recht \Gamma_p$ induces the $(p-1)^k$-fold covering $\pi : \cV \recht \cU$.

The subset $\cW \subset F_p^{4k}$ is the solution set of a system of three polynomial equations of degree at most $\ell$. If each of these polynomials is identically zero, we get that $\cW = F_p^{4k}$ and thus $\cU = (\Gamma_p)^k$. This means that $w$ is a law of $\Gamma_p$, which we supposed not to be the case. So at least one of the polynomials is not identically zero. The number of zeros of such a polynomial is bounded above by $\ell p^{4k-1}$ (and a better, even optimal, bound can be found in \cite{Se89}). So, $|\cW| \leq \ell p^{4k-1}$. Then also $|\cV| \leq \ell p^{4k-1}$ and because $\pi$ is a $(p-1)^k$-fold covering, we find that
$$|\cU| \leq \ell \, (p-1)^{-k} \, p^{4k-1} \; .$$
Since $|\Gamma_p| = (p-1) \, p \, (p+1)$, we conclude that
$$\rP\bigl( \, w(g_1,\ldots,g_k) = e \;\text{in}\; \Gamma_p \, \bigr) = \frac{|\cU|}{|\Gamma_p|^k} \leq \frac{\ell}{p} \, (p-1)^{-2k} \, (p+1)^{-k} \, p^{3k} \leq \frac{\ell}{p} \, \bigl(1+\frac{1}{p-1}\bigr)^{3k} \; .$$
So, \eqref{eq.est-l} holds.

Now assume that $p_n$ is a sequence of prime numbers and $k_n$ are positive integers such that $p_n \recht \infty$ and $\log k_n / \log p_n \recht 0$. For all $n$ large enough, $3k_n \leq p_n-1$ and for all $n$ large enough, as we explained in the beginning of the proof, $\Gamma_{p_n}$ has no law of length $\leq \ell$. Since $(1+1/x)^x < 3$ for all $x > 0$ and since there are less than $(2k)^{\ell +1}$ reduced words of length $\leq l$ in $g_1^{\pm 1},\ldots,g_k^{\pm 1}$, we find that for all $n$ large enough and a uniform, independent choice of $g_1,\ldots,g_{k_n} \in \Gamma_{p_n}$, we have
$$\rP\bigl( \; g_1,\ldots,g_{k_n} \;\;\text{satisfy a relation of length $\leq \ell$ in $\Gamma_{p_n}$} \; \Bigr) \leq (2k_n)^{\ell + 1} \; \frac{3 \ell}{p_n} \; .$$
By our assumption that $\log k_n / \log p_n \recht 0$, the right hand side tends to $0$ as $n \recht \infty$ and the lemma is proved.
\end{proof}

Combining Lemmas \ref{lem.gap} and \ref{lem.girth}, we obtain the following.

\begin{lemma} \label{lem.PGL}
For all $\eps > 0$ and all $k_0,p_0,\ell \in \N$, there exist a prime number $p \geq p_0$, an integer $k \geq k_0$ and elements $g_1,\ldots,g_k \in \Gamma_p = \PGL_2(\Z/p\Z)$ generating the group $\Gamma_p$ such that
\begin{enumlist}
\item $\|\pi(g_1,\ldots,g_k)\| \leq \eps$ for every nontrivial irreducible representation $\pi \in \Irr(\Gamma_p)$,
\item $g_1,\ldots,g_k$ satisfy no relation of length $\leq \ell$.
\end{enumlist}
\end{lemma}
\begin{proof}
Choose any sequence of prime numbers $p_n$ tending to infinity. Define $k_n = \lfloor(\log p_n)^2\rfloor$. Since $|\Gamma_{p_n}| = (p_n-1) \, p_n \, (p_n + 1)$, we get that $k_n / \log |\Gamma_{p_n}| \recht \infty$. Also, $\log k_n / \log p_n \recht 0$. So Lemmas \ref{lem.gap} and \ref{lem.girth} apply and for a large enough choice of $n$, properties 1 and 2 in the lemma hold for $p=p_n$, $k=k_n$ and a large portion of the $k_n$-tuples $(g_1,\ldots,g_{k_n}) \in \Gamma_{p_n}^{k_n}$.

The first property in the lemma is equivalent with
$$\Bigl\| \Bigl(\frac{1}{k} \sum_{j=1}^k \lambda(g_j)\Bigr)_{\ell^2(\Gamma_p) \ominus \C 1} \Bigr\| \leq \eps \; ,$$
where $\lambda : \Gamma_p \recht \ell^2(\Gamma_p)$ is the regular representation. If $\eps < 1$, it then follows in particular that there are no non zero functions in $\ell^2(\Gamma_p) \ominus \C 1$ that are invariant under all $\lambda(g_j)$, meaning that every element of $\Gamma_p$ can be written as a product of elements in $\{g_1,\ldots,g_k\}$. So, we get that $g_1,\ldots,g_k$ generate $\Gamma_p$.
\end{proof}

Having proven Lemma \ref{lem.PGL}, we are now ready to prove Theorem \ref{thm.free-c0}.

\begin{proof}[Proof of Theorem \ref{thm.free-c0}]
As in \eqref{eq.mean}, for every finite group $G$, subset $F \subset G$ and unitary representation $\pi : G \recht \cU(H)$, we write
$$\pi(F) := \frac{1}{|F|} \sum_{g \in F} \pi(g) \; .$$

For every prime number $p$, we write $\Gamma_p = \PGL_2(\Z/p\Z)$. We construct by induction on $n$ a sequence of prime numbers $p_n$ and a generating set
$$F_n \subset K_n := \prod_{j=1}^n \Gamma_{p_j}$$
such that, denoting by $\theta_{n-1} : K_n \recht K_{n-1}$ to projection onto the first $n-1$ coordinates, the following properties hold.
\begin{enumlist}
\item $\theta_{n-1}(F_n) = F_{n-1}$ and the map $\theta_{n-1} : F_n \recht F_{n-1}$ is an $r_n$-fold covering with $r_n \geq 2$.
\item If $\pi \in \Irr(K_n)$ and $\pi$ does not factor through $\theta_{n-1}$, then $\|\pi(F_n)\| \leq 1/n$.
\item The elements of $F_n$ satisfy no relation of length $\leq n$.
\end{enumlist}

Assume that $p_1,\ldots,p_{n-1}$ and $F_1,\ldots,F_{n-1}$ have been constructed. We have to construct $p_n$ and $F_n$. Write $k_1 = |F_{n-1}|$ and put $k_0 = \max \{ 2n+1, k_1 \}$. By Lemma \ref{lem.PGL}, we can choose $k_2 > k_0$, a prime number $p_n$ and a subset $F \subset \Gamma_{p_n}$ with $|F| = k_2$ such that the elements of $F$ satisfy no relation of length $\leq 3n$ and such that $\|\pi(F)\| \leq 1/(4n)$ for every nontrivial irreducible representation $\pi$ of $\Gamma_{p_n}$.

Write $F_{n-1} = \{g_1,\ldots,g_{k_1}\}$ and $F = \{h_1,\ldots,h_{k_2}\}$. Note that we have chosen $k_2 > \max\{2n+1,k_1\}$. So we can define the subset $F_n \subset K_{n-1} \times \Gamma_{p_n} = K_n$ given by
$$F_n = \{(g_i, h_i h_j h_i^{-1}) \mid 1 \leq i \leq k_1 \; , \; 1 \leq j \leq k_2 \; , \; i \neq j \} \; .$$
Note that $\theta_{n-1}(F_n) = F_{n-1}$ and that the map $\theta_{n-1} : F_n \recht F_{n-1}$ is a $(k_2 - 1)$-fold covering.

Every irreducible representation $\pi \in \Irr(K_n)$ that does not factor through $\theta_{n-1}$ is of the form $\pi = \pi_1 \ot \pi_2$ with $\pi_1 \in \Irr(K_{n-1})$ and with $\pi_2$ being a nontrivial irreducible representation of $\Gamma_{p_n}$. Note that
$$\pi(F_n) = \frac{1}{k_1} \sum_{i=1}^{k_1} \bigl( \pi_1(g_i) \ot \pi_2(h_i) \, T_i \, \pi_2(h_i)^* \bigr) \; ,$$
where
$$T_i := \frac{1}{k_2-1} \; \sum_{1 \leq j \leq k_2 \, , \, j \neq i} \; \pi_2(h_j) \; .$$
For every fixed $i \in \{1,\ldots,k_1\}$, we have
$$T_i = \frac{k_2}{k_2-1} \pi_2(F) - \frac{1}{k_2-1} \pi_2(h_i) \; .$$
Therefore,
\begin{equation}\label{eq.intermediate}
\|T_i\| < 2 \, \|\pi_2(F)\| + \frac{1}{2n} \leq \frac{1}{n} \; .
\end{equation}
It then also follows that $\|\pi(F_n)\| < 1/n$.

We next prove that $F_n$ is a generating set of $K_n$. Fix $i \in \{1,\ldots,k_1\}$. For all $s,t \in \{1,\ldots,k_2\}$ with $s \neq i$ and $t\neq i$, we have
$$(g_i, h_i h_s h_i^{-1}) \, (g_i, h_i h_t h_i^{-1})^{-1} = (e , h_i \, h_s h_t^{-1} \, h_i^{-1}) \; .$$
It thus suffices to prove that the set $H_i := \{h_s h_t^{-1} \mid s,t \in \{1,\ldots,k_2\} \setminus \{i\} \}$ generates $\Gamma_{p_n}$ for each $i \in \{1,\ldots,k_1\}$.

Denote by $\lambda_0$ the regular representation of $\Gamma_{p_n}$ restricted to $\ell^2(\Gamma_{p_n}) \ominus \C 1$. Define
$$R_i = \frac{1}{k_2-1} \; \sum_{1 \leq j \leq k_2 \, , \, j \neq i} \; \lambda_0(h_j) \; .$$
By \eqref{eq.intermediate}, we get that $\|R_i\| < 1$. Then also $\|R_i R_i^*\| < 1$. So, there is no non zero function in $\ell^2(\Gamma_{p_n}) \ominus \C 1$ that is invariant under all $\lambda(h)$, $h \in H_i$. It follows that each $H_i$ is a generating set of $\Gamma_{p_n}$.

Denote by $\eta_n : K_n \recht \Gamma_{p_n}$ the projection onto the last coordinate. If the elements of $F_n$ satisfy any relation of length $\leq n$, applying $\eta_n$ will give a nontrivial relation of length $\leq 3n$ between the elements of $F$. Since such relations do not exist, we have proved that the elements of $F_n$ satisfy no relation of length $\leq n$.

Define $K = \prod_{n=1}^\infty \Gamma_{p_n}$ and still denote by $\theta_n : K \recht K_n$ the projection onto the first $n$ coordinates. Define
$$F = \{k \in K \mid \theta_n(k) \in F_n \;\;\text{for all}\; n \geq 1 \} \; .$$
Note that $F \subset K$ is closed and $\theta_n(F) = F_n$. Denoting by $\langle F \rangle$ the subgroup of $K$ generated by $F$, we get that $\theta_n(\langle F \rangle) = K_n$ for all $n$. So, $\langle F \rangle$ is dense in $K$, meaning that $F$ topologically generates $K$.

Since each map $\theta_{n-1} : F_n \recht F_{n-1}$ is an $r_n$-fold covering, there is a unique probability measure $\nu_0$ on $K$ such that $(\theta_n)_*(\nu_0)$ is the normalized counting measure on $F_n$ for each $n$. Since $r_n \geq 2$ for all $n$, the measure $\nu_0$ is non atomic. Note that the support of $\nu_0$ equals $F$. Define the symmetric probability measure $\nu$ on $K$ given by $\nu(\cU) = (\nu_0(\cU) + \nu_0(\cU^{-1}))/2$ for all Borel sets $\cU \subset K$. The support of $\nu$ equals $F \cup F^{-1}$. Since $\lambda(\nu) = (\lambda(\nu_0) + \lambda(\nu_0)^*)/2$, to conclude the proof of the theorem, it suffices to prove that $F$ is free and that $\nu_0$ is a $c_0$ probability measure.

Let $g_1^{\eps_1} \cdots g_m^{\eps_m}$ be a reduced word of length $m$ with $g_1,\ldots,g_m \in F$. Take $n \geq m$ large enough such that $\theta_n(g_i) \neq \theta_n(g_{i+1})$ whenever $g_i \neq g_{i+1}$. We then get that $\theta_n(g_1)^{\eps_1} \cdots \theta_n(g_m)^{\eps_m}$ is a reduced word of length $m \leq n$ in the elements of $F_n$. It follows that
$$e \neq \theta_n(g_1)^{\eps_1} \cdots \theta_n(g_m)^{\eps_m} = \theta_n\bigl( g_1^{\eps_1} \cdots g_m^{\eps_m} \bigr) \; .$$
So, $g_1^{\eps_1} \cdots g_m^{\eps_m} \neq e$ and we have proven that $F$ is free.

We finally prove that $\|\pi(\nu_0)\| < 1/m$ for every irreducible representation $\pi$ of $K$ that does not factor through $\theta_m : K \recht K_m$. Since there are only finitely many irreducible representations that do factor through $\theta_m : K \recht K_m$, this will conclude the proof of the theorem. Let $\pi$ be such an irreducible representation. There then exists a unique $n > m$ such that $\pi = \pi_0 \circ \theta_n$ and $\pi_0$ is an irreducible representation of $K_n$ that does not factor through $\theta_{n-1} : K_n \recht K_{n-1}$. But then $\pi(\nu_0) = \pi_0(F_n)$ and thus
$$\|\pi(\nu_0)\| = \|\pi_0(F_n)\| \leq \frac{1}{n} < \frac{1}{m} \; .$$
\end{proof}


\end{document}